\documentclass[11pt]{amsart}
\usepackage{verbatim, graphicx, rotating}
\usepackage{mdwlist}
\usepackage{enumerate,mdwlist}
\usepackage{url}
\usepackage{amsthm}
\usepackage{hyperref}
\usepackage[pdf]{pstricks}
\usepackage{pst-node}
\usepackage{bbm}

\usepackage {amssymb}

\input xy
\xyoption{all}
\input epsf

\newlength{\tabwidth}
\newlength{\tabheight}
\newlength{\tabrule}
\newlength{\tabwidthx}
\newlength{\tabheightx}

\def\gentabbox#1#2#3#4{\vbox to \tabheight{\setlength{\tabrule}{#3}%
  \setlength{\tabwidthx}{#1\tabwidth}\addtolength{\tabwidthx}{\tabrule}%

\setlength{\tabheightx}{#2\tabheight}\addtolength{\tabheightx}{-\tabheight}%
  \hbox to #1\tabwidth{%
    \hspace{-0.5\tabrule}\rule{\tabrule}{#2\tabheight}\hspace{-\tabrule}%
    \vbox to #2\tabheight{\hsize=\tabwidthx%
      \vspace{-0.5\tabrule}\hrule width\tabwidthx height\tabrule%
      \vspace{-0.5\tabrule}\vfil%
      \hbox to \tabwidthx{\hss#4\hss}%
        \vfil\vspace{-0.5\tabrule}%
      \hrule width\tabwidthx height\tabrule\vspace{-0.5\tabrule}}%
    \hspace{-\tabrule}\rule{\tabrule}{#2\tabheight}\hspace{-0.5\tabrule}}%
  \vspace{-\tabheightx}}}
\def\genblankbox#1#2{\vbox to \tabheight{\vfil\hbox to
#1\tabwidth{\hfil}}}
\def\tabbox#1#2#3{\gentabbox{#1}{#2}{0.4pt}{\strut #3}}

\catcode`\:=13 \catcode`\.=13 \catcode`\;=13 
\catcode`\>=13 \catcode`\^=13
\def:#1\\{\hbox{$#1$}}
\def.#1{\tabbox{1}{1}{$#1$}}
\def>#1{\tabbox{2}{1}{$#1$}}
\def^#1{\tabbox{1}{2}{$#1$}}
\def;{\genblankbox{1}{1}\relax}
\catcode`\:=12 \catcode`\.=12 \catcode`\;=12 
\catcode`\>=12 \catcode`\^=7

\newcommand{\field}{\mathbb}
\newcommand{\liealgebra}{\mathfrak}
\newcommand{\la}{\liealgebra}

\newcommand{\C}{{\field C}}
\newcommand{\R}{{\field R}}
\newcommand{\Z}{{\field Z}}
\newcommand{\N}{{\field N}}

\newcommand{\Q}{{\mathbb Q}}

\renewcommand{\b}{\liealgebra b}

\newcommand{\n}{{\la n}}

\newcommand{\SL}{\mathrm{SL}}

\newtheorem{prop}{Proposition}[section]

\newtheorem{lemma}[prop]{Lemma}
\newtheorem{theorem}[prop]{Theorem}

\newtheorem{proposition}[prop]{Proposition}
\newtheorem{conjecture}[prop]{Conjecture}

\theoremstyle{definition}

\newtheorem{remark}[prop]{Remark}
\newtheorem{example}[prop]{Example}

\newtheorem*{question}{Question}

\newtheorem{definition}[prop]{Definition}

\newcommand{\caD}{\mathcal{D}}

\newcommand{\caG}{\mathcal{G}}
\newcommand{\caH}{\mathcal{H}}

\newcommand{\caM}{\mathcal{M}}

\newcommand{\caO}{\mathcal{O}}

\newcommand{\id}{\mathrm{id}}

\begin{document}
\title[$(1,2,1,2)$-avoiding $K$-orbits on $G/B$]
{Combinatorial results on $(1,2,1,2)$-avoiding $GL(p,\C) \times GL(q,\C)$-orbit closures on $GL(p+q,\C)/B$}

\author{Alexander Woo}
\thanks{AW is partially supported by NSA Young Investigators Grant H98230-13-1-0242}
\email{awoo@uidaho.edu}
\address{Alexander Woo \\
Department of Mathematics \\
Unversity of Idaho \\
875 Perimeter Drive, MS 1103 \\
Moscow, ID 83844-1103 \\
USA}
\author{Benjamin J. Wyser}
\thanks{BW is supported by NSF International Research Fellowship 1159045}
\email{Ben.Wyser@ujf-grenoble.fr}
\address{Ben Wyser \\
Institut Fourier - UMR 5582 \\
100 Rue des Math\'{e}matiques - BP 74 \\
38402 St. Martin-d'h\`{e}res \\
France}

\date{January 13, 2014}


\begin{abstract}
Using recent results of the second author which explicitly identify the ``$(1,2,1,2)$-avoiding" $GL(p,\C) \times GL(q,\C)$-orbit closures on the flag manifold $GL(p+q,\C)/B$ as certain Richardson varieties, we give combinatorial criteria for determining smoothness, lci-ness, and Gorensteinness of such orbit closures.  (In the case of smoothness, this gives a new proof of a theorem of W.M. McGovern.)  Going a step further, we also describe a straightforward way to compute the singular locus, the non-lci locus, and the non-Gorenstein locus of any such orbit closure.

We then describe a manifestly positive combinatorial formula for the Kazhdan-Lusztig-Vogan polynomial $P_{\tau,\gamma}(q)$ in the case where $\gamma$ corresponds to the trivial local system on a $(1,2,1,2)$-avoiding orbit closure $Q$ and $\tau$ corresponds to the trivial local system on any orbit $Q'$ contained in $\overline{Q}$.  This combines the aforementioned result of the second author, results of A. Knutson, the first author, and A. Yong, and a formula of Lascoux and Sch\"{u}tzenberger which computes the ordinary (type $A$) Kazhdan-Lusztig polynomial $P_{x,w}(q)$ whenever $w \in S_n$ is cograssmannian.
\end{abstract}

\maketitle

\section{Introduction}\label{sec:introduction}

Let $G$ be a complex semisimple reductive group and $B$ a Borel subgroup of $G$.  The opposite Borel subgroup $B^-$ (as well as $B$ itself) acts on the flag variety $G/B$ with finitely many orbits.  The closures of the orbits are known as the (opposite) Schubert varieties.  These orbits are indexed by the Weyl group $W$ for $G$, so for each $w\in W$, we have a Schubert variety $X^w$.  Especially in the case $G=\SL(n,\C)$, frequently known as the ``type A case'', the geometry of Schubert varieties has been extensively studied, both for its intrinsic interest and for its applications to the representation theory of $G$.

Particularly interesting for our present purposes are results that relate the geometry of $X^w$ to the combinatorics of the indexing element $w$ using the combinatorial notion of pattern avoidance.  Abe and Billey have recently written an excellent survey of such results~\cite{Abe-Billey-14}, of which we mention only a few here:

\begin{itemize}
	\item Pattern avoidance criteria have been given in all types to determine which Schubert varieties are smooth and which are rationally smooth \cite{LS-90,Billey-98,Billey-Postnikov-05}.
	\item In type $A$, the singular locus of a Schubert variety is completely understood in terms of patterns appearing in the Weyl group element \cite{Billey-Warrington-03,Cortez-03,Manivel-01,KLR-03}.
	\item In type $A$, criteria in terms of a generalization of pattern avoidance have been given to determine which Schubert varieties are Gorenstein \cite{Woo-Yong-06}.
	\item In type $A$, pattern avoidance criteria have been given to determine which Schubert varieties are local complete intersections (``lci") \cite{Ulfarsson-Woo-11}.
\end{itemize}

Now let $\theta$ be an involution of $G$, and let $K=G^\theta$ be the fixed points of the involution.  The pair $(G,K)$ is known as a \textit{symmetric pair}.  The subgroup $K$ also acts with finitely many orbits on the flag variety $G/B$~\cite{Matsuki-79}.  Moreover, natural combinatorial indexing sets for the orbits have been determined in many cases.  The geometry of $K$-orbits and their closures are important in the representation theory of the real forms $G_{\R}$ of $G$.

In one special case, the relationship between the geometry of the $K$-orbit and the combinatorics of its indexing element is well understood.  Let our group be $G\times G$ and our involution $\theta$ be given by $\theta(x,y)=(y,x)$.  In this case, the symmetric pair is $(G\times G, G)$, and the action is the diagonal action of $G$ on $G/B\times G/B$.  The orbit closures are fiber bundles with Schubert variety fibers over the smooth base $G/B$, so many geometric properties of these orbit closures can be understood by reference to analogous results for Schubert varieties.  Hence we refer to this case loosely as ``the Schubert case''.

However, outside of the Schubert case, the interaction between the geometry of the $K$-orbit closure and the combinatorics of the indexing set has been less well studied.  We know only of a few recent results in this direction:
\begin{itemize}
	\item Smoothness and rational smoothness have been characterized by McGovern and McGovern--Trapa in terms of pattern avoidance for various symmetric pairs \cite{McGovern-09a-arXiv,McGovern-09b,McGovern-11,McGovern-Trapa-09}.
	\item Graph-theoretic criteria for rational smoothness have been given by A. Hultman \cite{Hultman-12} for pairs $(G,K)$ satisfying fairly strong hypotheses --- these apply in particular to the symmetric pair $(G,K) = (GL(2n,\C),Sp(2n,\C))$.
\end{itemize}
To the authors' knowledge, no explicit description of the (rationally) singular locus of an orbit closure has been given in any of these cases, nor have any combinatorial criteria been given to categorize $K$-orbit closures with respect to more subtle singularity properties, such as Gorensteinness or lci-ness, for any symmetric pair.

Aside from these purely geometric questions, in the Schubert case, one can hope for combinatorial descriptions of Kazhdan--Lusztig (KL) polynomials.  In 1979, Kazhdan and Lusztig \cite{KL-79} introduced this family of polynomials $P_{v,w}(q)$, where $v$ and $w$ are elements of a Coxeter group $W$.  Defined by a recursion on $W$ or alternatively by certain axioms on elements of the Hecke algebra, these polynomials carry important information about the representation theory of both semisimple complex reductive groups and Coxeter groups.  Geometrically, in the case where $W$ is a Weyl group, their coefficients are dimensions of local intersection (co)homology groups of Schubert varieties, so in particular they are nonnegative.  (A proof for this nonnegativity that does not rely on the geometric interpretation, and hence holds generally for all Coxeter groups, was only recently given by Elias and Williamson~\cite{Elias-Williamson}.)  There have been numerous papers giving various combinatorial formulas for various classes of Kazhdan--Lusztig polynomials; some such papers are~\cite{LS-81, Lascoux-95, Billera-Brenti, Billey-Warrington-01, Jones-09}.

Inspired by questions in the representation theory of real reductive groups, Vogan and Lusztig--Vogan \cite{Vogan-83,LV-83} defined a more general family of polynomials associated to a symmetric pair $(G,K)$, now known as Kazhdan--Lusztig--Vogan (KLV) polynomials $P_{\tau,\gamma}(q)$.  (Indeed, the KLV polynomials for the Schubert case are the ordinary KL polynomials.)

In the most general case, $\tau$ and $\gamma$ are pairs, each consisting of a $K$-orbit $Q_{\tau}$ (or $Q_{\gamma}$) on $G/B$ together with a $K$-equivariant local system on the orbit.  In this paper, matters are simplified by the fact that for the pair $(G,K)$ that we consider, each orbit admits only the trivial $K$-equivariant local system.  Thus the reader may think of our KLV polynomials as being of the form $P_{\tau,\gamma}(q)$ where $Q_{\tau}$ and $Q_{\gamma}$ are $K$-orbits.

As with ordinary KL polynomials, the KLV polynomials can be defined by a recursion on the indexing set for local systems or alternatively by certain axioms on elements of a particular module over the Hecke algebra.  Like the ordinary KL-polynomials, their coefficients have a geometric interpretation as the dimensions of local intersection (co)homology groups of $K$-orbit closures, so in particular they are non-negative.  However, the authors know of only one previous result~\cite{Collingwood-92} implying a formula for KLV polynomials for a very special case where the $K$-orbit is a Schubert variety.

In this paper, we answer various combinatorial questions of the above flavor for certain $K=GL(p,\C) \times GL(q,\C)$-orbit closures on $GL(p+q,\C)/B$.  Using recent work of the second author \cite{Wyser-12b} identifying a certain subset of the $K$-orbit closures (those ``whose clans avoid $(1,2,1,2)$" --- this is explained in Section \ref{sec:preliminaries}) as certain Richardson varieties, we recover most of the results of \cite{McGovern-09a-arXiv} regarding pattern avoidance criteria for rational smoothness.  We also give combinatorial characterizations describing which orbit closures are lci and which are Gorenstein.  We moreover describe a diagrammatic procedure to calculate the singular locus, the non-lci locus, and the non-Gorenstein locus of a $(1,2,1,2)$-avoiding orbit closure.

The remaining $K$-orbit closures --- those whose clans \textit{contain} $(1,2,1,2)$ --- are more difficult to study.  They are not Richardson varieties but are instead cut out from Richardson varieties by certain projection conditions \cite[Theorem 2.5]{Wyser-13}.   (Equivalently, they can be described as the intersections of Richardson varieties and certain Hessenberg varieties.)  Therefore, the techniques used in this paper do not apply directly to them.  Nonetheless, it is natural to wonder about combinatorial characterizations of lci-ness and Gorensteinness of these more complicated orbit closures.  (The question of (rational) smoothness is understood; see Proposition \ref{prop:1212-rationally-singular}.)  We do not give complete answers to these questions, but we do briefly describe ways to approach them computationally, giving some partial evidence and conjectures based on our own experiments.

Next, combining the aforementioned work of the second author, recent work of A. Knutson, the first author, and A. Yong \cite{Knutson-Woo-Yong-13} reducing local questions on Richardson varieties to similar questions on Schubert varieties, and an old result of Lascoux and Sch\"utzenberger~\cite{LS-81} giving Kazhdan--Lusztig polynomials for cograssmannian Schubert varieties, we give explicit combinatorial formulas for the KLV polynomials associated to pairs of orbits $\tau,\gamma$, where $\gamma$ is an orbit whose clan avoids $(1,2,1,2)$.   This subsumes as a special case a result of Collingwood~\cite{Collingwood-92}.

We remark briefly on a possible extension to the results of this paper.  The symmetric pairs $(G,K) = (Sp(2n,\C),GL(2n,\C))$ and $(SO(2n,\C),GL(n,\C))$ are like $(GL(p+q,\C),GL(p,\C) \times GL(q,\C))$ in that, in each case, $K$ is the Levi subgroup of a minuscule parabolic subgroup of $G$.  This fact implies that a number of the $K$-orbit closures in those cases also coincide with Richardson varieties.  The second author has explicitly identified the $K$-stable Richardson varieties in these cases~\cite{Wyser-11b}.  (They again correspond to ``$(1,2,1,2)$-avoiding clans.")  The singular locus of minuscule Schubert varieties is described in \cite{LW-90}, so one can in principle recover some of the pattern avoidance results of \cite{McGovern-Trapa-09} for the case $(SO(2n,\C),GL(n,\C))$ and give similar results for the pair $(Sp(2n,\C),GL(n,\C))$.  Additionally, the Gorenstein locus of such Schubert varieties are described in \cite{Perrin-09}, so one should also be able to characterize Gorensteinness of $(1,2,1,2)$-avoiding orbit closures in these cases.  Finally, KL polynomials for minuscule Schubert varieties were given by B. Boe in~\cite{Boe-88}, so one can also use our methods to write down explicit formulas for some of the KLV polynomials for these symmetric pairs.  Finding these results will require understanding the translation between the combinatorics of (co)minuscule quotients of Weyl groups on the one hand and the combinatorics of the clan parameterization for these symmetric pairs on the other.  We will not attempt to make these additional translations explicit in this paper.

\section{Preliminaries}\label{sec:preliminaries}
\subsection{Notation and conventions}
In this paper, $G$ will denote the group $GL(n,\C)$.  We use $K$ to denote the symmetric subgroup $GL(p,\C) \times GL(q,\C)$.  Note that $K=G^{\theta}$ for $\theta = \text{int}(I_{p,q})$, where
\[  I_{p,q} := 
\begin{pmatrix}
I_p & 0 \\
0 & -I_q \end{pmatrix} \]
and where $\text{int}(g)$ denotes conjugation by $g$.  Realized in this way, $K$ is embedded in $G$ as the subgroup of block-diagonal matrices consisting of an upper-left invertible $p \times p$ block, a lower-right invertible $q \times q$ block, and zeros elsewhere.

$B$ and $B^-$ will denote the opposite Borel subgroups of $G$ consisting of upper and lower-triangular matrices, respectively, while $T = B \cap B^-$ will be the diagonal maximal torus of $G$.  The flag variety is isomorphic to $G/B$, with the coset $gB$ corresponding under this isomorphism to the complete flag whose $i$th subspace is the linear span of the first $i$ columns of $g$.

Let $u \in S_n$ be given as a permutation matrix in $G$.  For us, the \textbf{Schubert cell}, denoted $X_0^u$, will be the $B^-$-orbit $B^- uB/B$ of the $T$-fixed point $uB$.  The corresponding \textbf{Schubert variety} $X^u$ is the closure $\overline{X_0^u}$.  The Schubert variety $X^u$ is a subvariety of $G/B$ of complex codimension $l(u)$, where $l$ denotes the Coxeter length function on $S_n$.

Similarly, the \textbf{opposite Schubert cell} $X_u^0$ will be the $B$-orbit $BuB/B$, while the \textbf{opposite Schubert variety} $X_u$ will be the closure $\overline{X_u^0}$.  The opposite Schubert variety $X_u$ is a subvariety of $G/B$ of complex dimension $l(u)$.  (Note that many papers reverse our definitions of Schubert and opposite Schubert cells and varieties.)

Left multiplication by the longest element $w_0\in S_n$ gives an isomorphism between the opposite Schubert variety $X_u$ and the Schubert variety $X^{w_0u}$.  Also, given $u\in S_n$, we will sometimes refer to the point $uB/B$ simply as $u$ or as $p_u$.

\subsection{Combinatorial parameters for $K=GL(p,\C) \times GL(q,\C)$-orbits on $GL(p+q,\C)/B$}\label{ssec:clans}
The finitely many $K$-orbits on $G/B$ are customarily indexed by \textit{$(p,q)$-clans}, as described in, for example, \cite{Matsuki-Oshima-90,Yamamoto-97,McGovern-Trapa-09}.  We now recall the details of this indexing.
\begin{definition}
A \textbf{$(p,q)$-clan} is a string $\gamma=(c_1, \hdots ,c_n)$ of $n=p+q$ symbols, each of which is a $+$, a $-$, or a natural number.  The string must satisfy the following two properties:
\begin{enumerate}
	\item Every natural number which appears must appear exactly twice in the string.
	\item The difference between the number of plus signs and the number of minus signs in the string must be $p-q$.  (If $q > p$, then there should be $q-p$ more minus signs than plus signs.)
\end{enumerate}
\end{definition}

We only consider such strings up to an equivalence relation saying that only the positions of matching natural numbers, rather than the actual values of the numbers, are necessary to determine the clan.  For instance, the clans $(1,2,1,2)$, $(2,1,2,1)$, and $(5,7,5,7)$ are all the same, since they all have matching natural numbers in positions $1$ and $3$ and also in positions $2$ and $4$.  On the other hand, $(1,2,2,1)$ is a different clan, since it has matching natural numbers in positions $1$ and $4$ and in positions $2$ and $3$.

A theorem of \cite{Matsuki-Oshima-90}, elaborated upon in \cite{Yamamoto-97}, gives an explicit bijection between the set of $(p,q)$-clans and the set of $K$-orbits on $G/B$:

\begin{theorem}[\cite{Matsuki-Oshima-90,Yamamoto-97}]\label{thm:orbit_description}
There is an explicit bijection between the set of $(p,q)$-clans and the $K$-orbits on $G/B$.
\end{theorem}

For each $(p,q)$-clan $\tau$, we need an explicit point of the corresponding orbit $Q_{\tau}$.  We now outline an algorithm, described in \cite{Yamamoto-97}, which produces certain such representatives, which we call {\bf Yamamoto points} of $Q_{\tau}$.

First, for each pair of matching natural numbers of $\tau$, we assign one of the numbers a ``signature" of $+$, and the other a signature of $-$.  We next choose a permutation $v$ of $1,\hdots,n$ whose one-line notation places $1,\hdots,p$ (in any order) in the positions of the $+$ signs and numbers assigned a signature of $+$, and $p+1,\hdots,n$ (in any order) in the remaining positions.

Having determined such a permutation $\sigma$, let $F_{\bullet} = \left\langle v_1, \hdots, v_n \right\rangle$ to be the flag specified as follows:
\[ v_i = 
\begin{cases}
	e_{v(i)} + e_{v(j)} & \text{ if $c_i \in \N$, $c_i$ has signature $+$, and $c_i = c_j$}, \\
	e_{v(i)} & \text{ otherwise.}
\end{cases} \]

Then $F_{\bullet} \in Q_{\tau}$, the $K$-orbit corresponding to the clan $\tau$.

Note that the algorithm described above allows for several choices.  We describe a particularly natural scheme for these choices.  To each pair of numbers, assign the first a signature of $+$ and the second a signature of $-$.  Then choose $v$ to be the permutation whose one-line notation places the numbers $1,\hdots,p$ in ascending order on the positions of the $+$ signs and the first occurrences of natural numbers, and whose one-line notation places the numbers $p+1,\hdots,n$, also in ascending order, on the remaining positions.

\begin{definition}
We call the Yamamoto point of $Q_{\tau}$ obtained using this choice of permutation $v$ the \textbf{distinguished} representative of $Q_{\tau}$.
\end{definition}

We give two examples.  For the $(3,3)$-clan $\tau=(+,+,+,-,-,-)$, the permutation $v=123456$, and the distinguished representative is the standard flag,
\[ \left\langle e_1,\hdots,e_6 \right\rangle. \]

Now, let $\tau$ be the $(2,2)$-clan $(1,-,+,1)$.  The permutation $v$ is $1324$, and the distinguished representative is
\[ \left\langle e_1 + e_4, e_3, e_2, e_4 \right\rangle. \]

Given a clan $\tau$, we now associate to it two Grassmannian permutations $v(\tau)$ and $u(\tau)$.  The permutation $v(\tau)$ is the inverse of the permutation $v$ which we have just described.  Explicitly, its one-line notation is formed by first listing in ascending order the positions of $\tau$ containing a $+$ or the first occurrence of a number, then listing in ascending order the positions with a $-$ or the second occurrence of a number.

The permutation $u(\tau)$ is obtained in a similar way.  Its one-line notation is formed by first listing in ascending order the positions of $\tau$ which have a $+$ or the \textit{second} occurrence of a number, followed by listing in ascending order the positions with a $-$ or the \textit{first} occurence of a number.

For example, if $\tau=(1,2,+,-,1,2)$, then $v(\tau) = 123456$, and $u(\tau) = 356124$.  If $\tau=(1,2,2,3,3,1)$, then $v(\tau) = 124356$ and $u(\tau) = 356124$.

Now, let $u$ be the permutation obtained from $v=v(\tau)^{-1}$ and $\tau$ as follows.  For every pair of matching natural numbers $c_i = c_j \in \N$ of $\tau$, interchange the entries of the one-line notation for $v$ in positions $i$ and $j$.  Returning to the two examples above, if $\tau=(1,2,+,-,1,2)$, then $v=123456$, while $u=563412$.  If $\tau=(1,2,2,3,3,1)$, then $v=124356$, and $u=642531$.

Then we have the following easy result.

\begin{lemma}\label{lem:t-fixed-representative}
Let $\tau$ be a $(p,q)$-clan, and let $F_{\bullet}$ be the distinguished representative of $Q_{\tau}$.  Let $u,v$ be as just defined.  Then $F_{\bullet}$ is in the $B^-$-orbit of the $T$-fixed point $vB/B$ and the $B$-orbit of the $T$-fixed point $uB/B$.  In other words, $F_{\bullet} \in X_u^0 \cap X_0^v$.

Moreover, although $u^{-1} \neq u(\tau)$ in general, $u^{-1}$ is in the same left $W_K = S_p \times S_q$-coset as $u(\tau)$.
\end{lemma}
\begin{proof}
Let $\tau=(\tau_1,\hdots,\tau_n)$.  The flag $F_{\bullet}$ is of the form $gB$, where $g$ is the matrix whose columns are the vectors $v_i$ given by the Yamamoto algorithm above.  Evidently, $g$ is almost the permutation matrix with $1$'s in positions $(v(i),i)$, except that in each column corresponding to an index $i$ such that $\tau_i \in \N$ is a first occurrence, there is an extra $1$ in row $v(j)$, where $\tau_i = \tau_j$.  Note that by our choice of $v$, $v(j) > v(i)$, so the $1$ in position $(v(j),i)$ occurs further down in column $i$ than that in position $(v(i),i)$.  Thus, using the left $B^-$-action by downward row operations, we may eliminate these extra $1$'s, giving the point $vB/B$.

For the second claim, we first change our matrix representative for $F_\bullet$ (or equivalently present some of the vector spaces in $F_\bullet$ with a different basis).  Using the right action of $B$ by rightward column operations, we can first eliminate the second $1$ from any row consisting of two $1$'s.  There is one such row for each pair of matching natural numbers $\tau_i = \tau_j \in \N$ ($i<j$), namely row $u(i)$.  The effect of such a column operation is to eliminate the second $1$ in position $(u(i),j)$ and move it instead (up) to position $(u(j),j)$.  As a result, row $u(j)$ now has two $1$'s, one at position $(u(j),i)$ and the other at position $(u(j),j)$, whereas row $u(i)$ has only one $1$ in position $(u(i),i)$.

Now, using the left $B$-action by upward row operations on this new representative, we can eliminate the additional $1$ in position $(u(j),i)$.  Doing so for all pairs $\tau_i = \tau_j$ gives the point $uB/B$.

For the last claim, note that by construction, $u$ has a one-line notation in which $1,\hdots,p$ (in some order) are on the $+$'s and second occurrences, and $p+1,\hdots,n$ (in some order) are on the $-$'s and first occurrences.  Left-multiplying $u$ by some element of $S_p \times S_q$ will put both sets in ascending order.  We then take the inverse to obtain the desired result.
\end{proof}

There is a natural notion of pattern avoidance for $(p,q)$-clans, used first by McGovern in \cite{McGovern-09a-arXiv}.  We say that one clan $\gamma$ \textbf{contains} another clan $\gamma'$ if there are character positions within $\gamma$ which, when extracted from $\gamma$ in order, give a clan equivalent to $\gamma'$, where the equivalence is, as described above, up to permutation of the natural numbers.  We say that $\gamma$ \textbf{avoids} the pattern $\gamma'$ if it does not contain it.  In particular, note that $\gamma$ ``avoids the pattern $(1,2,1,2)$" if any two pairs of matching natural numbers of $\gamma$ are either nested or disjoint.  So, for example, $(1,1,2,2,3,3)$ avoids $(1,2,1,2)$, but $(1,2,1,3,2,3)$ does not.  
  
The first main theorem that we use in our combinatorial analysis of some of the $K$-orbit closures is a result of \cite{Wyser-12b} identifying $(1,2,1,2)$-avoiding $K$-orbit closures explicitly as certain Richardson varieties.  Associated to a $(1,2,1,2)$-avoiding clan $\gamma$, we have the Grassmannian permutations (each with at most one descent at position $p$) $u(\gamma)$ and $v(\gamma)$ defined a few paragraphs ago.  Then we have the following theorem.

\begin{theorem}[{\cite[Theorem 3.8 \& Remark 3.9]{Wyser-12b}}]\label{thm:richardson-theorem}
Given any $(p,q)$-clan $\gamma$ avoiding the pattern $(1,2,1,2)$, let $Q_{\gamma}$ be the associated $K$-orbit.  Let $u = w_0^K u(\gamma)^{-1}$, and let $v = v(\gamma)^{-1}$, where $w_0^K$ denotes the long element of $W_K = S_p \times S_q$ which reverses the sets $1,\hdots,p$ and $p+1,\hdots,n$.  Then $\overline{Q_{\gamma}}$ is the Richardson variety $X_u^v = X_u \cap X^v$.
\end{theorem}

\subsection{KL and KLV polynomials}\label{ssec:klv-defs}
In this section, we quickly recall the definitions --- first algebraic, then geometric --- of both the ordinary Kazhdan-Lusztig (KL) polynomials and the Kazhdan--Lusztig--Vogan (KLV) polynomials.  The latter polynomials were originally defined by Vogan \cite{Vogan-83} for a general symmetric pair $(G,K)$.  We give the general definition then explicitly explain how to calculate KLV polynomials for $(GL(n,\C), GL(p,\C) \times GL(q,\C))$ in terms of the combinatorics of clans.  This case is simpler than the general case but more complicated than the cases treated by Hultman~\cite[Sect. 5]{Hultman-12}.  Our hope is that the explicit description will be helpful to combinatorialists interested in understanding these polynomials.

\subsubsection{Combinatorial definitions}
First, recall the definition of the ordinary KL polynomials.  Given a Weyl group $W$ with simple reflections $S$, the Hecke algebra $\mathcal{H}_W$ is the $\Z[q^{\pm 1/2}]$-algebra with basis $T_w$ for $w\in W$ and multiplication defined by 
\[ T_sT_w=
\begin{cases}
T_{sw} & \text{if $sw > w$} \\
(q-1)T_w + qT_{sw} & \text{ if $sw < w$.}
\end{cases}
\]

The Hecke algebra has a ring involution defined by $\overline{q^{1/2}}=q^{-1/2}$ and $\overline{T_w}=(T_{w^{-1}})^{-1}$.  We can define the $R$-polynomials by
$$\overline{T_w}=q^{-\ell(w)}\sum_{v}(-1)^{\ell(v)}R_{v,w}(q)T_v;$$
if we do so, then $R_{v,w}(q)$ will be polynomials in $q$ with $R_{w,w}=1$ for all $w$ and $R_{v,w}=0$ whenever $v\not\leq w$.  Indeed, the involution can be defined as the unique ring homomorphism satisfying $\overline{q^{1/2}}=q^{-1/2}$ and this condition on $R_{v,w}$.

In \cite{KL-79}, Kazhdan-Lusztig showed that there exists a unique basis $\{C^\prime_w\}_{w\in W}$ satisfying the following:
\begin{enumerate}
\item $\overline{C^\prime_w}=C^\prime_w$ for all $w\in W$.
\item If we define $P_{x,w}(q)$ by 
\[ C^\prime_w=q^{-\ell(w)/2}\displaystyle\sum_{x \leq w} P_{x,w}(q)T_x, \]
then $P_{x,w}(q)$ is a uniquely determined polynomial in $q$, provided that we insist that
\begin{enumerate}
	\item $P_{w,w}(q)=1$, and
	\item $\deg(P_{x,w}) \leq \frac{1}{2}(\ell(w)-\ell(x)-1)$.
\end{enumerate}
\end{enumerate}

From these facts and the base cases $C^\prime_1=1$ and $C^\prime_s=q^{-1/2}(1+T_s)$ for all simple reflections $s$, one can recursively calculate the Kazhdan--Lusztig elements $C^\prime_w$ and hence the \textbf{Kazhdan-Lusztig (KL) polynomials} $P_{x,w}(q)$.  If $ws>w$, then
\[ C^\prime_wC^\prime_s=C^\prime_{ws}+\displaystyle\sum_{v<ws} E_v(q)C^\prime_v, \]
where $E_v(q)$ is explicitly either $0$ if $vs<v$ or else the coefficient of $q^{(\ell(w)-\ell(v)-1)/2}$ in $P_{v,w}(q)$.  However, the explicit description of $E_v(q)$ is not necessary, since we can recursively determine $E_v(q)$ purely from the degree bound (2b).  In particular, if the coefficients in $C^\prime_wC^\prime_s$ satisfy the degree bound (2b), then $E_v(q)=0$ for all $v$ and $C^\prime_{ws}=C^\prime_wC^\prime_s$.

KLV polynomials for a symmetric pair $(G,K)$ have a similar definition in terms of an $\caH_W$-module $\caM_K$.  Let $\mathcal{D}$ consist of all pairs $(Q,\delta)$, where $Q$ is a $K$-orbit on $G/B$ and $\delta$ is a $K$-equivariant local system on $Q$.  Since $Q$ is determined by the local system $\delta$, we will write $\delta$ to mean $(Q,\delta)$.  The module $\caM_K$ is free over $\Z[q^{\pm 1/2}]$ with basis $\{\mathbf{T}_\delta\}_{\delta\in\mathcal{D}}$.  We will not describe the action of $\caH_W$ on $\caM_K$ in general, but we later give an explicit description of this action in terms of clans when $K=GL(p,\C) \times GL(q,\C)$ and $G=GL(p+q,\C)$.

On the set $\caD$ there is a partial order called \textit{Bruhat $\caG$-order}, defined in \cite{Vogan-83}.  We indicate this order by $<$.  Like Bruhat order on $W$, Bruhat $\caG$-order on $\mathcal{D}$ is graded by a length function $\ell$.  (Bruhat $\caG$-order is roughly defined by inclusion of orbits, complicated by the possibility that multiple local systems can be associated to a single orbit.  The length $\ell(\delta)$ is the dimension of the orbit associated to $\delta$ minus the minimal dimension for all orbits.\footnote{The original definitions in~\cite{Vogan-83} defined length as simply the dimension of the orbit, but it is easy to see that adding a constant to all lengths has no effect as long as it is done consistently.  We use our definition both for simplicity and to agree with the Atlas of Lie Groups software.})  We can now define an involution on $\caM_K$ by requiring that
\begin{enumerate}
\item $\overline{h\cdot m}=\overline{h}\cdot\overline{m}$ for all $h\in\caH_W$ and all $m\in\caM_K$.
\item If we define $R_{\gamma,\delta}$ by 
\[ \overline{\mathbf{T}_\delta}=(-q^{-\ell(\delta)})\sum (-1)^{\ell(\gamma)}R_{\gamma,\delta}(q)\mathbf{T}_\delta, \]
then $R_{\gamma,\delta}(q)=0$ unless $\ell(\gamma)\leq\ell(\delta)$, and $R_{\delta,\delta}(q)=1$ for all $\delta$.
\end{enumerate}

Given the bar involution, the KLV polynomials $P_{\gamma,\delta}$ can be defined exactly as the KL polynomials are.  There is a unique basis $\{\mathbf{C}^\prime_\delta\}_{\delta\in\mathcal{D}}$ satisfying the following:
\begin{enumerate}
\item $\overline{\mathbf{C}'_\delta}=\mathbf{C}'_{\delta}$ for all $\delta\in\mathcal{D}$.
\item If we define $P_{\gamma,\delta}(q)$ by 
\[ \mathbf{C}'_\delta=q^{-\ell(\delta)/2}\sum_\tau P_{\gamma,\delta}(q)\mathbf{T}_\delta, \]
then $P_{\gamma,\delta}(q)$ is a uniquely determined polynomial in $q$, provided we insist that
\begin{enumerate}
	\item $P_{\delta,\delta}(q)=1$, and
	\item $\deg(P_{\gamma,\delta}) \leq \frac{1}{2}(\ell(\delta)-\ell(\gamma)-1)$.
\end{enumerate} 
\end{enumerate}

For many groups $K$ (including $K=GL(p,\C) \times GL(q,\C)$), one can also recursively compute $\mathbf{C}^\prime_\delta$ as for KL polynomials.  The base cases are given by $\mathbf{C}^\prime_\delta=\mathbf{T}_\delta$ whenever $\delta$ is a minimal element in Bruhat $\mathcal{G}$-order.  Otherwise, for each $\delta\in\mathcal{D}$, one finds $\tau\in\mathcal{D}$ with $\tau<\delta$ and a simple reflection $s\in W$ such that $\mathbf{T}_\delta$ is the unique maximal basis element (in Bruhat $\mathcal{G}$-order) occuring in the expansion of $C^\prime_s\mathbf{C}^\prime_\tau$.  Then one recursively computes 
\[ C^\prime_s\mathbf{C}^\prime_\tau=\mathbf{C}^\prime_\delta+\displaystyle\sum_{\gamma<\delta} E_\gamma(q)\mathbf{C}^\prime_\gamma, \]
where $E_\gamma(q)$ is explicitly either $0$ or a coefficient of $P_{\tau,\gamma}(q)$ depending on certain relations in Bruhat $\mathcal{G}$-order.  Again, $E_\gamma(q)$ can be recursively determined from the degree bound (2b), and in particular, if the coefficients in $C^\prime_s\mathbf{C}^\prime_\tau$ satisfy the degree bound (2b), then $E_\gamma(q)=0$ for all $\gamma$, and $\mathbf{C}^\prime_\delta=C^\prime_s\mathbf{C}^\prime_\tau$.

(Unfortunately, for some groups there are local systems $\delta$, not minimal in Bruhat $\mathcal{G}$-order, for which nevertheless no such $s$ and $\tau$ exist.  See~\cite{Atlas-computation-orig,Atlas-computation-improved} for an explanation of how to perform this calculation in that case as well as further details of this computation.  Unfortunately, they do not use the combinatorial language of clans.)

For the pair $(GL(p+q,\C),GL(p,\C) \times GL(q,\C))$, as mentioned above, matters are simplified substantially by the fact that each $K$-orbit admits only the trivial $K$-equivariant local system.  Thus each element of $\mathcal{D}$ can be thought of as simply a $K$-orbit, and the Bruhat $\mathcal{G}$-order amounts to inclusion of orbits in orbit closures.  We remark that for the pair $(SL(p+q,\C),S(GL(p,\C) \times GL(q,\C))$, the orbit set is precisely the same as for the pair $(GL(p+q,\C),GL(p,\C) \times GL(q,\C))$, 
but here some orbits \textit{do} admit non-trivial $K$-equivariant local systems if $p=q$.  In such
cases, our results still give the KLV polynomials $P_{\tau,\gamma}$ whenever $\tau$ and $\gamma$ are trivial local
systems on the corresponding orbits, but there are other KLV polynomials about which we cannot say anything.  (Note that $K$ and the geometry of its orbits depends on the specific form of the Lie group, so $(PGL(p+q,\C),P(GL(p,\C)\times GL(q,\C)))$ is yet another separate case with a different orbit set to which our results do not apply at all.) 

We now describe the action of $\caH_W$ on $\caM_K$ for the pair $(GL(p+q,\C),GL(p,\C) \times GL(q,\C))$ by describing $T_{s_i}\mathbf{T}_\gamma$ for each simple transposition $s_i$ and each clan $\gamma$.  Given a clan $\gamma=(\gamma_1,\hdots,\gamma_n)$, let $\gamma_i$ denote the $i$-th entry.  Also, let $\gamma\times s_i$ denote the clan which is obtained from $\gamma$ by switching $\gamma_i$ and $\gamma_{i+1}$.  Finally, we need a formula for the \textbf{length} of a clan, denoted $\ell(\gamma)$, which is given in~\cite{Yamamoto-97} as
$$\ell(\gamma)=\sum_{c_i=c_j\in\mathbb{N}, i<j} (j-i-\#\{k\in\mathbb{N}\mid c_s=c_t=k \text{ for some } s<i<t<j\}).$$

\begin{itemize}
\item[(compact imaginary)] If $\gamma_i=\gamma_{i+1}=+$, or $\gamma_i=\gamma_{i+1}=-$, then $T_{s_i} \mathbf{T}_\gamma=q \mathbf{T}_\gamma$.
\item[(noncompact imaginary)] If $\gamma_i$ and $\gamma_{i+1}$ are opposite signs, then $T_{s_i}\mathbf{T}_\gamma =\mathbf{T}_{\gamma'}+\mathbf{T}_{\gamma\times s_i}$, where $\gamma'$ is obtained from $\gamma$ by changing $\gamma_i$ and $\gamma_{i+1}$ to an unused natural number.
\item[(real)] If $\gamma_i$ and $\gamma_{i+1}$ are mates, then $T_{s_i} \mathbf{T}_\gamma=(q-2)\mathbf{T}_\gamma+(q-1)\mathbf{T}_{\gamma'}+(q-1)\mathbf{T}_{\gamma''}$, where $\gamma'$ and $\gamma''$ are obtained from $\gamma$ by changing $\gamma_i$ and $\gamma_{i+1}$ to one $+$ and one $-$ in either order.
\item[(complex ascent)] If we are not in the above cases, then $\ell(\gamma\times s_i)=\ell(\gamma)\pm1$.  If $\ell(\gamma\times s_i)=\ell(\gamma)+1$, then $T_{s_i}\mathbf{T}_\gamma=\mathbf{T}_{\gamma\times s_i}$.
\item[(complex descent)] If $\ell(\gamma\times s_i)=\ell(\gamma)-1$, then $T_{s_i}\mathbf{T}_\gamma=q\mathbf{T}_{\gamma\times s_i}+(q-1)\mathbf{T}_\gamma$.
\end{itemize}

The labels on the cases correspond to the classification of roots for a $\theta$-stable torus in \cite{Vogan-83}, and the multiplication rules are translations of the specific rules of~\cite[Lemma 3.5]{LV-83}.  The real and noncompact imaginary cases are ``type I'' in Vogan's classification; ``type II'' cases do not occur for $(GL(p+q,\C),GL(p,\C)\times GL(q,\C))$.  The reader may note that the rules in the cases of a complex ascent or descent are similar to the multiplication rules in the Hecke algebra, while the others do not occur in that case.

We can be more precise about distinguishing complex ascents from descents.  By the \textbf{mate} of a natural number entry of a clan, we mean the other entry with the same natural number.
\begin{itemize}
\item If $\gamma_i$ is a number and $\gamma_{i+1}$ is a sign, then $s_i$ is a complex ascent if the mate of $\gamma_i$ is to the left and a complex descent otherwise.
\item If $\gamma_i$ is a sign and $\gamma_{i+1}$ is a number, then $s_i$ is a complex ascent if the mate of $\gamma_{i+1}$ is to the right and a complex descent otherwise.
\item If $\gamma_i$ and $\gamma_{i+1}$ are different numbers, then $s_i$ is a complex ascent if the mate of $\gamma_i$ occurs to the left of the mate of $\gamma_{i+1}$.
\end{itemize}

Combinatorially speaking, it would be more appropriate to consider $\mathcal{M}_K$ as a right module rather than a left module as we have written above, since $s_i$ acts on clans by permuting positions, not entries (though it is not clear what permuting entries would mean).  It would also be more desirable from a geometric viewpoint (at least when considering $K$-orbits on $G/B$) to consider $\mathcal{M}_K$ as a right module, but we bow to historical convention.

As a simple example of our description of KLV polynomials, we now calculate $\mathbf{C}^\prime_{(1,2,1,2)}$.  We can use the recursion with
$$C^\prime_{s_2}\mathbf{C}^\prime_{(1,1,2,2)} = \mathbf{C}^\prime_{(1,2,1,2)}+\sum E_\gamma(q)\mathbf{C}^\prime_\gamma.$$
It will turn out that the coefficients of $C^\prime_{s_2}\mathbf{C}^\prime_{(1,1,2,2)}$ satisfy the degree bound, so $E_\gamma(q)=0$ for all $\gamma$.

We know that $$C^\prime_{s_2}=q^{-1/2}(T_{s_2}+T_1),$$ and, since the orbit closure for $(1,1,2,2)$ is smooth (or by further recursive calculation), we have that
\begin{align*}
\mathbf{C}^\prime_{(1,1,2,2)}&=q^{-2/2}(\mathbf{T}_{(1,1,2,2)}+\mathbf{T}_{(+,-,1,1)}+\mathbf{T}_{(-,+,1,1)}
+\mathbf{T}_{(1,1,+,-)} +\mathbf{T}_{(1,1,-,+)}\\
&+\mathbf{T}_{(+,-,+,-)}+\mathbf{T}_{(-,+,-,+)}+\mathbf{T}_{(+,-,-,+)} +\mathbf{T}_{(-,+,+,-)}.
\end{align*}

When multiplying these terms by $T_{s_2}$, we are in the compact imaginary case for $(+,-,-,+)$ and $(-,+,+,-)$, the noncompact imaginary case for $(-,+,-,+)$ and $(+,-,+,-)$, and the complex ascent case for the remaining terms.  Hence,
\begin{align*}
T_{s_2}(\mathbf{T}_{(+,-,-,+)}+\mathbf{T}_{(-,+,+,-)}) &= q(\mathbf{T}_{(+,-,-,+)}+\mathbf{T}_{(-,+,+,-)}), \\
T_{s_2}(\mathbf{T}_{(-,+,-,+)}+\mathbf{T}_{(+,-,+,-)}) &= (\mathbf{T}_{(-,+,-,+)}+\mathbf{T}_{(+,-,+,-)}
+\mathbf{T}_{(+,-,-,+)}+\mathbf{T}_{(-,+,+,-)}),
\end{align*}
and putting the entire product together,
\begin{align*}
C^\prime_{s_2}\mathbf{C}^\prime_{(1,1,2,2)}&=q^{-3/2} (\mathbf{T}_{(1,2,1,2)}+\mathbf{T}_{(+,1,-,1)}
+\mathbf{T}_{(-,1,+,1)}+\mathbf{T}_{(1,+,1,-)}\\
&+\mathbf{T}_{(1,-,1,+)}+\mathbf{T}_{(1,1,2,2)}+\mathbf{T}_{(+,-,1,1)} +\mathbf{T}_{(-,+,1,1)}+\mathbf{T}_{(1,1,+,-)} \\
&+\mathbf{T}_{(1,1,-,+)}+\mathbf{T}_{(+,-,+,-)}
+\mathbf{T}_{(-,+,-,+)}+(1+q)\mathbf{T}_{(+,-,-,+)}+(1+q)\mathbf{T}_{(-,+,+,-)})
\end{align*}

Since this expression satisfies the degree bound, it must in fact be $\mathbf{C}^\prime_{(1,2,1,2)}$, and $P_{(+,-,-,+),(1,2,1,2)}=P_{(-,+,+,-),(1,2,1,2)}=1+q$ while all other KLV polynomials $P_{\tau,(1,2,1,2)}$ are either $1$ or $0$ depending on whether or not $\tau\leq(1,2,1,2)$ in Bruhat order.

\subsubsection{Geometric interpretations}\label{ssec:geometric-klv}
First, recall the geometric interpretation of ordinary KL polynomials due to Kazhdan and Lusztig.  Given a variety $X$ and a point $p\in X$, let $IH^i_p(X)$ denote the $i$-th local intersection cohomology of $X$ at $p$.  In principle, one can calculate this as follows.  From an appropriate stratification of $X$, one constructs (as in \cite{GM-83}) a complex of sheaves $\mathcal{IH}(X)$ called the intersection cohomology sheaf.  One construction of $\mathcal{IH}(X)$ starts from the trivial local system on the largest stratum and extends it by certain truncations of derived pushforward (with compact support) on the derived category of sheaves on $X$~\cite{GM-83}.  Therefore, the complex $\mathcal{IH}(X)$ is also sometimes referred to as the Deligne-Goresky-MacPherson (DGM) extension of the trivial local system to $X$.  One can localize this complex at the point $p$, creating a complex of vector spaces.  The $i$-th cohomology of this complex is what we call $IH^i_p(X)$.  Kazhdan and Lusztig show in \cite{KL-80} that
\[ P_{v,w}(q)=\sum_i \dim IH^i_{vB/B}(X_w)q^{i/2}.\]

An analogous result holds for KLV polynomials in the $K$-orbit setting.  For the $(p,q)$-clan $\gamma$, denote by $Q_{\gamma}$ the corresponding $K$-orbit, and denote by $Y_{\gamma}$ the Zariski closure of $Q_{\gamma}$.  Abusing notation, let $\gamma$ also denote the trivial local system on $Q_{\gamma}$, and let $IH(\gamma)$ be the DGM extension of $\gamma$ to $Y_{\gamma}$.  Denote by $IH^i(\gamma)$ the $i$th cohomology of this complex.  For any $(p,q)$-clan $\tau$ with $Q_{\tau} \subseteq Y_{\gamma}$, denote by $[\tau:IH^i(\gamma)]$ the composition factor multiplicity of $\tau$ in $IH^i(\gamma)$ (in the category of $K$-equivariant constructible sheaves on $G/B$), where again we abuse notation and use $\tau$ to denote the trivial local system on $Q_{\tau}$.

Then the KLV polynomial $P_{\tau,\gamma}$ can be defined as follows \cite[Thm. 1.12]{LV-83}:
\[ P_{\tau,\gamma}(q)=\sum_i [\tau:IH^i(\gamma)] q^{i/2}. \]
In particular, all odd cohomology vanishes, as $P_{\tau,\gamma}(q)$ is an honest polynomial in $q$.

Localizing at a point $p \in Q_{\tau}$, we get
\[ P_{\tau,\gamma}(q)=\sum_i \dim IH^i_p(Y_{\gamma}) q^{i/2}. \]
Note that the left hand side should technically be the sum of all KLV polynomials $P_{\tau^\prime,\gamma}(q)$ where
$\tau^\prime$ runs over the set of all $K$-equivariant local systems on $Q_{\tau}$.
However, as we have mentioned, in our case no non-trivial $K$-equivariant local systems exist on any orbit.  
Hence the KLV polynomials we consider here are actually $IH$-Poincar\'{e} polynomials for $K$-orbit closures,
as is the case for ordinary KL polynomials and Schubert varieties.

\section{Combinatorial Criteria for Singularity Properties of $(1,2,1,2)$-avoiding Orbit Closures}\label{sec:path-criteria}
In this section, we use Theorem \ref{thm:richardson-theorem} to determine combinatorially which $(1,2,1,2)$-avoiding $K$-orbit closures possess certain singularity properties and which do not.

First we need some general facts about singularity properties on Richardson varieties.  We say a property $\mathcal{P}$ is {\bf local} if it is determined strictly by examining the local rings at points of the variety.  If a property $\mathcal{P}$ is local, we say it is {\bf open} if the $\mathcal{P}$-locus (meaning the set of points at which $X$ has the property) is an open set.  Furthermore, we say a property $\mathcal{P}$ is {\bf multiplicative} if it holds on $X\times Y$ precisely when it holds on both $X$ and $Y$.  Suppose that $\mathcal{P}$ is an open multiplicative local property of algebraic varieties. For example, being smooth, being a local complete intersection (lci), and being Gorenstein are all examples of such properties.  Then the following result on how to determine when a Richardson variety has property $\mathcal{P}$ is proved in \cite{Knutson-Woo-Yong-13}.

\begin{lemma}\label{lem:richardson-p-criterion}
Let $\mathcal{P}$ be an open multiplicative local property of algebraic varieties.
The Richardson variety $X_u^v$ has property $\mathcal{P}$ if and only if the Schubert variety $X^v$ is $\mathcal{P}$ at $u$ \textit{and} the opposite Schubert variety $X_u$ is $\mathcal{P}$ at $v$ (or equivalently, the Schubert variety $X^{w_0u}$ is $\mathcal{P}$ at $w_0v$).
\end{lemma}

We also require another easy, generally known lemma.  We include its proof for lack of a suitable reference.

\begin{lemma}\label{lem:inverse-property}
If $\mathcal{P}$ is an open multiplicative local property that holds for regular local rings, then the Schubert variety $X^v$ is $\mathcal{P}$ at $u$ if and only if the Schubert variety $X^{v^{-1}}$ is $\mathcal{P}$ at $u^{-1}$.
\end{lemma}
\begin{proof}
Denote by $\id$ the point $1B/B$.  Consider the diagonal action of $G$ on $G/B\times G/B$, and consider the $G$-orbit closure $\mathcal{Z}_v:=\overline{G\cdot (\id, v)}$, with surjective projection maps $\pi_1,\pi_2:\mathcal{Z}\rightarrow G/B$ onto the first and second factors respectively.

The fiber $\pi_1^{-1}(\id)$ is $\id \times X_v$, and $\pi_1$ is $G$-equivariant, so, taking an affine neighborhood $U$ of $\id$, we have $\pi_1^{-1}(U)=U \times X_v$.  Consider the point $(\id, u)\in \mathcal{Z}_v$.  It has an open neighborhood $U\times V$, where $V$ is isomorphic to an open neighborhood of $u$ in $X_v$.

On the other hand, we also have that $\mathcal{Z}_v=\overline{G\cdot(v^{-1}, \id)}$, so $\pi_2^{-1}(\id)=X_{v^{-1}}\times\id$.  Taking an affine neighborhood $U^\prime$ of $\id$, we have $\pi_2^{-1}(U^\prime)=X_{v^{-1}}\times U^\prime$.  Now consider the point $(u^{-1},\id)\in\mathcal{Z}_u$.  It has an open neighborhood $V^\prime\times U^\prime$, where $V^\prime$ is isomorphic to an open neighborhood of $u^{-1}$ in $X_{v^{-1}}$.

The points $(\id,v)$ and $(v^{-1},\id)$ are in the same $G$-orbit, and $\mathcal{Z}_u$ is $G$-invariant, so $\mathcal{P}$ holds on $U\times V$ if and only if it holds on $V^\prime\times U^\prime$.  Since $U$ and $U^\prime$ are smooth (since they are open subsets of $G/B$) and $\mathcal{P}$ is multiplicative, $\mathcal{P}$ holds on $V$ if and only if it holds on $V^\prime$.
\end{proof}

Let $\gamma$ be a $(1,2,1,2)$-avoiding $(p,q)$-clan, $Q_\gamma$ the corresponding $K$-orbit, and $Y_\gamma=\overline{Q_\gamma}$ its Zariski closure.  By Theorem \ref{thm:richardson-theorem}, $Y_\gamma=X_u^v$, where $u=(u(\gamma)w_0^K)^{-1}$ and $v=v(\gamma)^{-1}$.  To determine whether $Y_\gamma$ has property $\mathcal{P}$, by Lemma~\ref{lem:richardson-p-criterion}, it suffices to check whether $X^v$ is $\mathcal{P}$ at $u$ and whether $X^{w_0 u}$ is $\mathcal{P}$ at $w_0 v$.  Now by Lemma~\ref{lem:inverse-property}, it suffices to check whether $X^{v(\gamma)}$ is $\mathcal{P}$ at $u(\gamma)w_0^K$ and whether $X^{u(\gamma)w_0^Kw_0}$ is $\mathcal{P}$ at $v(\gamma)w_0$.  Note that, since $u(\gamma)$ is Grassmannian, so is $u(\gamma)w_0^Kw_0$.  Hence the determination of whether $Y_\gamma$ has property $\mathcal{P}$ boils down to checking whether certain points of Grassmannian Schubert varieties lie in the $\mathcal{P}$-locus.  For the properties $\mathcal{P}=$``(rationally) smooth", ``lci", and ``Gorenstein", the $\mathcal{P}$-locus of these special Schubert varieties is known.  Hence we are able to provide in what follows combinatorial criteria for these properties in the case of $(1,2,1,2)$-avoiding $\gamma$.

Note that, for general permutations $w$ and $x$, $X^{w_0ww_0}$ is isomorphic to $X^w$ by an isomorphism that takes the point $x$ to $w_0xw_0$.  Hence, checking whether $X^{u(\gamma)w_0^Kw_0}$ is $\mathcal{P}$ at $v(\gamma)w_0$ is equivalent to checking whether $X^{w_0u(\gamma)w_0^K}$ is $\mathcal{P}$ at $w_0v(\gamma)$.  By another application of Lemma~\ref{lem:richardson-p-criterion}, this is equivalent to checking if $\mathcal{P}$ holds on $X^{v(\gamma)}_{u(\gamma)w_0^K}$.  While this observation is not strictly necessary in what follows, we will frequently use it for brevity.

For the properties $\mathcal{P}$ listed above, the results stating when a permutation $u$ is in the $\mathcal{P}$-locus of a Grassmannian Schubert variety $X^v$ are best described in terms of a {\bf path diagram} associated to the permutations $u$ and $v$.  We now describe how to draw this diagram.  Let $p$ be the descent of the Grassmannian permutation $v$.  Start with a $p \times q$ rectangle, and trace a lattice path from the southwest corner to the northeast, moving either one unit right or one unit up at each step.  At the $i$th step, the path moves right if $v^{-1}(i) > p$, and up if $v^{-1}(i) \leq p$.  Note that this path determines a partition, namely the one whose Young diagram consists of the blocks of the $p \times q$ grid lying weakly northwest of it.  However, for our purposes, it is actually the path itself we are interested in.

As an example, consider the Grassmannian permutation $v=1367245$, which has a unique descent at position $4$.  The associated path fits inside a $4 \times 3$ rectangle as in Figure \ref{fig:first-path}.

\begin{figure}[h]
\centering
\begin{pspicture}(3,4)
\psgrid[subgriddiv=0,gridlabels=0]
\psdots[dotsize=0.2](0,0)(0,1)(1,1)(1,2)(2,2)(3,2)(3,3)(3,4)
\psline[linewidth=0.08](0,0)(0,1.05)
\psline[linewidth=0.08](0,1)(1.05,1)
\psline[linewidth=0.08](1,1)(1,2.05)
\psline[linewidth=0.08](1,2)(3.05,2)
\psline[linewidth=0.08](3,2)(3,4.05)
\end{pspicture}
\caption{The path of $v=1367245$}
\label{fig:first-path}
\end{figure}
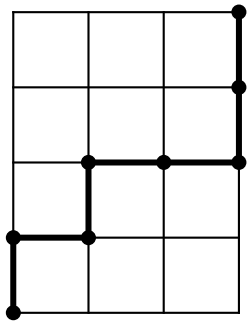

If $u \geq v$ is another permutation, then its path lies weakly southeast of that of $v$.  (Note that $u$ does not have to be Grassmannian, but when $w\in W_K$, the paths for $u$ and $uw$ are the same, and all local properties of $X^v$ are the same at both $u$ and $uw$.)  For instance, in Figure \ref{fig:second-path} the paths for the two Grassmannian permutations $v=124673589 < 156892347=u$ are shown drawn in the same $5 \times 4$ grid, with the path further southeast being that for $u$ and the one further northwest being that for $v$.

\begin{figure}[h]
\centering
\begin{pspicture}(4,5)
\psgrid[subgriddiv=0,gridlabels=0]
\psdots[dotsize=0.2](0,0)(0,1)(0,2)(1,2)(1,3)(2,3)(2,4)(2,5)(3,5)(4,5)(1,1)(2,1)(3,1)(3,2)(3,3)(4,3)(4,4)
\psline[linewidth=0.08](0,0)(0,2.05)
\psline[linewidth=0.08](0,2)(1.05,2)
\psline[linewidth=0.08](1,2)(1,3.05)
\psline[linewidth=0.08](1,3)(2.05,3)
\psline[linewidth=0.08](2,3)(2,5.05)
\psline[linewidth=0.08](2,5)(4.05,5)
\psline[linewidth=0.08](0,1)(3.05,1)
\psline[linewidth=0.08](3,1)(3,3.05)
\psline[linewidth=0.08](3,3)(4.05,3)
\psline[linewidth=0.08](4,3)(4,5.05)
\end{pspicture}
\caption{The paths for $v=124673589 < 156892347 = u$}
\label{fig:second-path}
\end{figure}
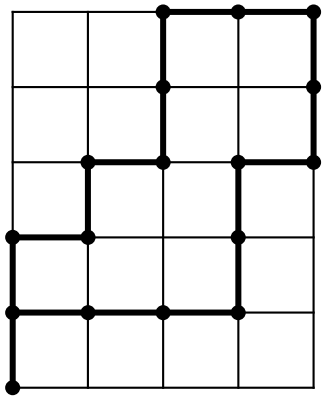

Thus the two paths for $u$ and $v$ determine a chain of skew shapes, each connected to the next at a point (or perhaps by a series of line segments, if the paths for $u$ and $v$ happen to coincide over some portion of the grid).  We refer to those skew shapes that are \textit{not} simply sequences of line segments --- those that actually open and then close, bounding a region of positive area --- as the \textbf{components} of the path diagram.  For each component of the path diagram, we call the portion of the path for $u$ which bounds its southeast side its \textbf{bottom boundary} and the portion of the path for $v$ which bounds its northwest side its \textbf{top boundary}.

For the purpose of convenience when we later recall the Lascoux-Sch\"utzenberger rule for KL-polynomials associated to Grassmannian permutations, we will prefer to draw the diagrams just described rotated clockwise $45^{\circ}$, and we will generally omit the portion of the $p \times q$ grid not lying along either path, drawing only the paths for $u$ and $v$ themselves.  Thus the above example of $u=156892347$ and $v=124673589$ will be depicted as in Figure \ref{fig:first-rotated-diagram}.  (Note that the path for $u$ now lies \textit{below} the path for $v$, so that the ``bottom" and ``top" boundaries of components are now appropriately named.)

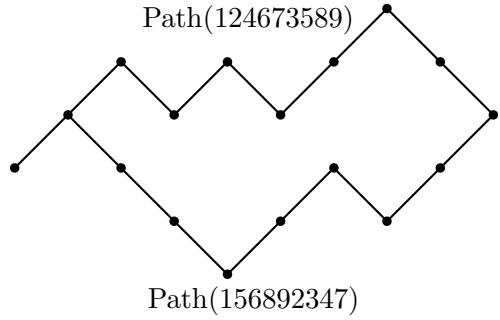
\begin{figure}[h]
\centering
\begin{pspicture}(4,5)
\rput[b]{-45}(0,2.5){
\psdots(0,0)(0,1)(0,2)(1,2)(1,3)(2,3)(2,4)(2,5)(3,5)(4,5)(1,1)(2,1)(3,1)(3,2)(3,3)(4,3)(4,4)
\psline(0,0)(0,2.05)
\psline(0,2)(1.05,2)
\psline(1,2)(1,3.05)
\psline(1,3)(2.05,3)
\psline(2,3)(2,5.05)
\psline(2,5)(4.05,5)
\psline(0,1)(3.05,1)
\psline(3,1)(3,3.05)
\psline(3,3)(4.05,3)
\psline(4,3)(4,5.05)
\rput[l]{*0}(2.5,0){Path(156892347)}
\rput[l]{*0}(-0.2,2.6){Path(124673589)}
}
\end{pspicture}
\caption{The rotated path diagram for $u=156892347$, $v=124673589$}
\label{fig:first-rotated-diagram}
\end{figure}

Note that, based on the algorithm for drawing the path diagram, it is easy to see that the path diagram for $vw_0\geq uw_0^Kw_0$ is simply the partition diagram for $u\geq v$ flipped upside down (with $p$ and $q$ also exchanged).

The combinatorial translation from a $(1,2,1,2)$-avoiding clan to a path diagram of this type is now easy to describe, using Theorem \ref{thm:richardson-theorem} and the above definitions.

\begin{definition}\label{def:path-diagram}
The \textbf{path diagram} for a $(1,2,1,2)$-avoiding $(p,q)$-clan $\gamma=(c_1, \hdots ,c_n)$ is drawn as follows:  Starting at the southwest corner of a $p \times q$ rectangle (rotated $45^{\circ}$) and tracing to the northeast corner, we draw two paths, $P_1$ and $P_2$, following these rules at step $i$ for $i=1,\hdots,n$:

\begin{enumerate}
	\item If $c_i=+$, both $P_1$ and $P_2$ move up;
	\item If $c_i=-$, both $P_1$ and $P_2$ move right;
	\item If $c_i$ is the first occurrence of a natural number, then $P_1$ moves up, while $P_2$ moves right;
	\item If $c_i$ is the second occurrence of a natural number, then $P_1$ moves right, while $P_2$ moves up.
\end{enumerate}
\end{definition}

It is clear in the above definition that $P_1$ (the upper path) is the path for $v(\gamma)$, while $P_2$ (the lower path) is the path for $u(\gamma)$.  Moreover, it is clear that the components of the path diagram open at the first occurrence of a natural number which is not contained within any other matching pair of numbers and close at the second occurrence of such a number.  We refer to such a matching pair as \textbf{outermost}.

\begin{example}
The path diagram for the $(6,5)$-clan $(1,+,-,+,2,+,-,-,2,+,1)$ is shown in Figure \ref{fig:rotated-clan-diagram}.
\begin{figure}[h]
\centering
\begin{pspicture}(5,6)
\rput[b]{-45}(0,2.5){
\psdots(0,0)(0,1)(0,2)(1,2)(1,3)(1,4)(1,5)(2,5)(3,5)(4,5)(4,6)(5,6)(1,0)(1,1)(2,1)(2,2)(3,2)(3,3)(4,3)(5,3)(5,4)(5,5)
\psline(0,0)(0,2)
\psline(0,2)(1,2)
\psline(1,2)(1,5)
\psline(1,5)(4,5)
\psline(4,5)(4,6)
\psline(4,6)(5,6)
\psline(0,0)(1,0)
\psline(1,0)(1,1)
\psline(1,1)(2,1)
\psline(2,1)(2,2)
\psline(2,2)(3,2)
\psline(3,2)(3,3)
\psline(3,3)(5,3)
\psline(5,3)(5,6)
}
\end{pspicture}
\caption{The rotated path diagram for $\gamma=(1,+,-,+,2,+,-,-,2,+,1)$}
\label{fig:rotated-clan-diagram}
\end{figure}
\end{example}

\subsection{(Rational) Smoothness}
\subsubsection{Globally (rationally) smooth $(1,2,1,2)$-avoiding $K$-orbit closures}
We start by determining which $(1,2,1,2)$-avoiding orbit closures are smooth.  Recall that a complex variety $X$ of dimension $n$ is \textbf{smooth at a point $p$} if the local ring $(\caO_{X,p},\mathfrak{m},\mathbbm{k})$ is \textbf{regular}, meaning that $\dim_{\mathbbm{k}} \mathfrak{m} / \mathfrak{m}^2$ is equal to the Krull dimension of $\caO_{X,p}$.  A variety $X$ is simply said to be \textbf{smooth} if it is smooth at every point.  Recall also that $X$ is \textbf{rationally smooth at $p$} if
\[ H^q(X,X \setminus \{p\}; \Q) \cong 
\begin{cases}
	\Q & \text{ if $q = 2n$} \\
	0 & \text{ otherwise}
\end{cases}
\]
and simply \textbf{rationally smooth} if it is rationally smooth at every point.

In general, smoothness and rational smoothness at a point are not equivalent notions, with smoothness being a strictly stronger condition.  However, for all points on type $A$ Schubert varieties, these conditions are known to be equivalent \cite{Deodhar-85}.  Since our checks of smoothness or rational smoothness of $(1,2,1,2)$-avoiding orbit closures reduce to checks for the same properties on two type $A$ Schubert varieties, the two conditions amount here to the same thing.  Thus we simply refer to the property of interest here as ``smoothness", dropping the redundant modifier ``rational".

For Grassmannian permutations $u$ and $v$, we describe how to use the path diagrams described above to decide
\begin{itemize}
	\item Whether $X^v$ is globally smooth, and
	\item Whether $X^v$ is smooth at $u$.
\end{itemize}
These criteria are well-known.  They appear explicitly in \cite{LW-90}, but are also implicit in earlier work such as \cite{LS-81} (when combined with the aforementioned equivalence of smoothness and rational smoothness established in \cite{Deodhar-85}) or \cite{Zelevinsky-83}.

We call a lattice point on the path for $v$ (the top one) an \textbf{outer corner} if the (unrotated) path contains both the lattice point directly south of it and the lattice point directly east of it.  Analogously, we call a lattice point on the path for $v$ an \textbf{inner corner} if the (unrotated) path contains both the lattice points directly north and directly west of it.  The by now classically known theorem is as follows.

\begin{proposition}
\label{prop:grass-sing-crit}
The Schubert variety $X_v$ is singular at $u$ if and only if there is at least one inner corner on the path for $v$ that is not on the path for $u$ or, equivalently, an inner corner on the top boundary of a component of the path diagram.
\end{proposition}

Note this implies that $X_v$ is singular (at some point) if and only if the path for $v$ has an inner corner within the strict interior of the $p\times q$ rectangle.

For brevity, we refer to an inner corner on the path for $v$ that is not on the path for $u$ as a \textbf{singular corner}.  For example, in Figure \ref{fig:singular-top}, $X^v$ is singular at $u$, with the open dot indicating the lone singular corner.

\begin{remark}
The above definitions of inner and outer corner given above are opposite of what is usually found in the literature.  This is because usually, ``inner" and ``outer" are relative to the Young diagram which lies northwest of the path.  However, here we are thinking of ``inner" and ``outer" relative to the interiors of the components of the path diagram for $\gamma$, which lie southeast of the path for $v$.
\end{remark}

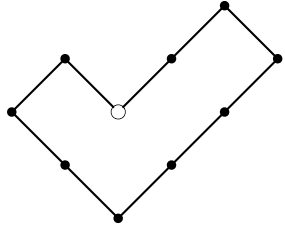
\begin{figure}[h]
\centering
\begin{pspicture}(2,3)
\rput[b]{-45}(0,1.5){
\psdots(0,0)(0,1)(1,2)(1,3)(2,3)(1,0)(2,0)(2,1)(2,2)
\psdot[dotstyle=o,dotsize=0.2](1,1)
\psline(0,0)(0,1)
\psline(0,1)(0.92,1)
\psline(1,1.08)(1,3)
\psline(1,3)(2,3)
\psline(0,0)(2,0)
\psline(2,0)(2,3)
}
\end{pspicture}
\caption{$X^v$ singular at $u$:  A singular corner on the top boundary}
\label{fig:singular-top}
\end{figure}

On the other hand, in Figure \ref{fig:smooth-top}, $X^v$ is smooth at $u$.

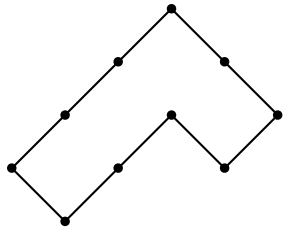
\begin{figure}[h]
\centering
\begin{pspicture}(2,3)
\rput[b]{-45}(0,0.5){
\psdots(0,0)(0,1)(0,2)(0,3)(1,3)(2,3)(1,0)(1,1)(1,2)(2,2)
\psline(0,0)(0,3)
\psline(0,3)(2,3)
\psline(0,0)(1,0)
\psline(1,0)(1,2)
\psline(1,2)(2,2)
\psline(2,2)(2,3)
}
\end{pspicture}
\caption{$X^v$ smooth at $u$:  No singular corners on the top boundary}
\label{fig:smooth-top}
\end{figure}

Recall that, to check smoothness of a $(1,2,1,2)$-avoiding $K$-orbit closure $Y_\gamma$, or equivalently smoothness of the Richardson variety $X^{v(\gamma)}_{u(\gamma)w_0^K}$, we must also check whether $X^{uw_0^Kw_0}$ is smooth at $vw_0$.  However, since the path diagram for this pair is simply the one for $X^v$ at $u$ flipped upside down, we simply have to perform the same check upside down.  
To be precise, $X^{uw_0^Kw_0}$ is singular at $vw_0$ if and only if there is an inner (meaning to the interior side of the path diagram) corner along the \textit{lower} path which does not lie on the \textit{upper} path or, equivalently, which lies on the bottom boundary of some component of the path diagram.  We also call a corner of this type a \textbf{singular corner}.  So for instance, in Figure \ref{fig:smooth-top}, $X^v$ \textit{is} smooth at $uw_0^K$, but the Richardson variety $X_{uw_0^K}^v$ is singular because $X^{w_0uw_0^K}$ is singular at $w_0v$.  The open dot in Figure \ref{fig:singular-bottom} marks the singular corner on the bottom boundary.

\begin{figure}[h]
\centering
\begin{pspicture}(2,3)
\rput[b]{-45}(0,0.5){
\psdots(0,0)(0,1)(0,2)(0,3)(1,3)(2,3)(1,0)(1,1)(2,2)
\psdot[dotstyle=o,dotsize=0.2](1,2)
\psline(0,0)(0,3)
\psline(0,3)(2,3)
\psline(0,0)(1,0)
\psline(1,0)(1,1.92)
\psline(1.08,2)(2,2)
\psline(2,2)(2,3)
}
\end{pspicture}
\caption{$X^{w_0uw_0^K}$ singular at $w_0v$:  A singular corner on the bottom boundary}
\label{fig:singular-bottom}
\end{figure}
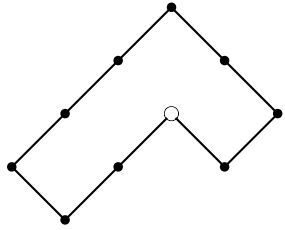

Given this pictorial characterization of smoothness of Richardson varieties of the type that we are interested in, we can now give the following pattern-avoidance criterion for smoothness of a $(1,2,1,2)$-avoiding $K$-orbit closure.

\begin{proposition}\label{prop:1212-avoiding-rationally-smooth}
Let $\gamma$ be a $(1,2,1,2)$-avoiding $(p,q)$-clan.  Then $Y_\gamma = \overline{Q_{\gamma}}$ is smooth if and only if $\gamma$ avoids the patterns $(1,+,-,1)$, $(1,-,+,1)$, $(1,+,2,2,1)$, $(1,-,2,2,1)$, $(1,2,2,+,1)$, $(1,2,2,-,1)$, and $(1,2,2,3,3,1)$.
\end{proposition}
\begin{proof}
As we have noted, singular corners occur only on either the top or bottom boundary of a component of the path diagram, so they are a result of characters of the clan occurring between an outermost pair of matching numbers.  More specifically, a singular corner on the top boundary of some component of the path diagram occurs if and only if there are two consecutive character positions $c_i$ and $c_{i+1}$ of $\gamma$ between the corresponding outermost pair such that one of the following is true:

\begin{enumerate}
	\item $c_i$ is a $-$, and $c_{i+1}$ is a $+$;
	\item $c_i$ is a $-$, and $c_{i+1}$ is the first occurrence of a natural number;
	\item $c_i$ is the second occurrence of a natural number, and $c_{i+1}$ is a $+$; or
	\item $c_i$ is the second occurrence of a natural number, and $c_{i+1}$ is the first occurrence of a (different) natural number.
\end{enumerate}

These four possibilities respectively imply that $\gamma$ contains the pattern $(1,-,+,1)$, $(1,-,2,2,1)$, $(1,2,2,+,1)$, or $(1,2,2,3,3,1)$.  Furthermore, there is a singular corner on the bottom of the path diagram if and only if there are consecutive character positions $c_i$ and $c_{i+1}$ of $\gamma$ occurring within the corresponding outermost pair and satisfying one of (1)-(4), except with inverted signs.  If $Y_\gamma$ is singular, the path diagram for $v(\gamma)$ and $u(\gamma)$ must have a singular corner.  Hence $\gamma$ must contain one of the bad patterns.

The other half of the proof is to show that, if $\gamma$ contains one of the bad patterns, then its path diagram contains a singular corner.  Supposing that $\gamma$ contains one of these patterns, we take the matching $1$'s of the pattern to be outermost.  Let $C$ be the component of the path diagram corresponding to this outermost pair.  Then one checks easily that in each case, either the bottom boundary of $C$ has an ``up" segment followed at some later point by a ``right" segment, or the top boundary of $C$ has a ``right" segment followed at some later point by an ``up" segment (or both).  In the former case, the bottom boundary of $C$ must change from ``up" to ``right" at some point, giving a singular corner.  In the latter case, the top boundary of $C$ must change from ``right" to ``up" at some point, again giving a singular corner.
\end{proof}

One would naturally wonder about $(1,2,1,2)$-\textit{containing} orbit closures as well.  In fact, the result here is as simple as one could hope for.

\begin{proposition}[\cite{McGovern-09a-arXiv}]\label{prop:1212-rationally-singular}
If $\gamma$ contains the pattern $(1,2,1,2)$, then $Y_\gamma = \overline{Q_{\gamma}}$ is rationally singular.
\end{proposition}

Proposition \ref{prop:1212-rationally-singular} is proved in \cite{McGovern-09a-arXiv} using a combinatorial result of Springer \cite{Springer-94} giving a root-theoretic necessary condition for a $K$-orbit closure to be rationally smooth.  We do not have anything to add to this portion of McGovern's argument.  However, combining Propositions \ref{prop:1212-avoiding-rationally-smooth} and \ref{prop:1212-rationally-singular}, we have given a new proof of the main result of \cite{McGovern-09a-arXiv}:

\begin{theorem}[\cite{McGovern-09a-arXiv}]\label{thm:mcgovern}
Let $\gamma$ be a $(p,q)$-clan.  The $K$-orbit closure $\overline{Q_{\gamma}}$ is rationally smooth if and only if it avoids the patterns $(1,2,1,2)$, $(1,+,-,1)$, $(1,-,+,1)$, $(1,+,2,2,1)$, $(1,-,2,2,1)$, $(1,2,2,+,1)$, $(1,2,2,-,1)$, and $(1,2,2,3,3,1)$.  Moreover, smoothness and rational smoothness of $K$-orbit closures are equivalent for $(GL(p+q,\C),GL(p,\C)\times GL(q,\C))$.\footnote{The version of this paper in the Journal of Algebra has an error in the statement of this theorem.  The cited version on arXiv is correct.}
\end{theorem}

\subsubsection{The singular locus of a $(1,2,1,2)$-avoiding $K$-orbit closure}\label{sec:sing-locus}
The results of \cite{McGovern-09a-arXiv} determine which $K$-orbit closures are singular but do not determine \textit{where} they are singular.  Here, we describe how to compute the singular locus of a $(1,2,1,2)$-avoiding $K$-orbit closure using its path diagram.

Let $\gamma$ be a $(p,q)$-clan, and let $Y_{\gamma} = \overline{Q_{\gamma}}$ be the corresponding $K$-orbit closure.  Let $S$ be the singular locus of $Y_{\gamma}$.  Since the left action of any element $k \in K$ takes an open neighborhood of any point $p \in Y_{\gamma}$ to an isomorphic open neighborhood of the point $k \cdot p \in Y_{\gamma}$, the singular locus of $Y_{\gamma}$ is $K$-stable.  Being closed, it is hence a union of $K$-orbit closures.  Thus a description of the singular locus amounts to giving the list of $K$-orbits (or clans) whose closures are the irreducible components of $S$.  Said another way, we wish to list those clans $\tau$ such that $Y_{\gamma}$ is singular along $Q_{\tau}$ and such that $\tau$ is (Bruhat) maximal with this property.

The following proposition, from~\cite[Corollary 1.3]{Knutson-Woo-Yong-13} describes the singular locus of a Richardson variety in terms of the singular loci of Schubert varieties.

\begin{proposition}
Let $\Sigma(X)$ denote the singular locus of a variety $X$.  Then
$$\Sigma(X^v_u)=(\Sigma(X^v)\cap X_u)\cup(\Sigma(X_u)\cap X^v).$$
\end{proposition}

It follows that the singular locus of a Richardson variety is a union of Richardson varieties.  Thus if $\gamma$ is $(1,2,1,2)$-avoiding, then by the previous paragraph we know that $S$ is a union of $K$-stable Richardson varieties, or in other words, closures of $(1,2,1,2)$-avoiding $K$-orbits.  Thus in searching for those $\tau$ described in the previous paragraph, we can restrict our attention to $(1,2,1,2)$-avoiding clans.

Bruhat order on $(1,2,1,2)$-avoiding orbit closures is determined solely by containment of path diagrams.  Moreover, if $\tau < \gamma$ are $(1,2,1,2)$-avoiding, then by the above proposition and Proposition~\ref{prop:grass-sing-crit}, $Y_{\gamma}$ is singular along $Y_{\tau}$ if and only if $\gamma$ has a singular corner (either on the top or the bottom path) that does not lie on the path diagram for $\tau$.  Thus the path diagrams for the $\tau$ we seek are precisely the largest ones missing a singular corner.

We construct these path diagrams by removing hooks from the skew diagrams bounded by the path diagram $D_{\gamma}$ for $\gamma$.  More specifically, we have one diagram $D_{\tau}$ for each singular corner of $D_{\gamma}$, formed as follows.  If the singular corner is on the bottom path, then we look at the corner box immediately above it and remove its hook, taking $D_{\tau}$ to be the boundary of the resulting skew shape.  If the singular corner is on the top path, then we do the same for the box immediately below it.  Clearly, the resulting shapes are precisely those that both miss a singular corner and are maximal with respect to containment among shapes having this property.

Each such diagram $D_{\tau}$ can easily be converted back to the clan $\tau$, as follows.  A path diagram does not necessarily specify a clan, but it \textit{does} specify what we will call the \textbf{FS-pattern} for $\tau$, which is a sequence of $+$'s, $-$'s, F's, and S's, with $\pm$ appearing wherever $\tau$ has one of these symbols, and F (respectively, S) appearing wherever $\tau$ has a first occurrence (respectively, a second occurrence).  To obtain the FS-pattern, we simply note, for each $i$, what the two paths do at step $i$.  If they both move up, character $i$ is a $+$.  If they both move right, character $i$ is a $-$.  If the top path moves up while the bottom path moves right, then character $i$ is an F.  Finally, if the top path moves right while the bottom path moves up, then character $i$ is an S.

Now, it is possible for multiple clans to have the same FS-pattern --- for instance, $(1,2,1,2)$ and $(1,2,2,1)$ have the same FS-pattern $(F,F,S,S)$, yet they are different clans.  However, by the above discussion, we know that the $\tau$ we seek is $(1,2,1,2)$-avoiding, and there \textit{is} a unique $(1,2,1,2)$-avoiding clan with a given FS-pattern.  To compute it, we simply move from left to right and insist that every second occurrence be matched with the most recently appearing first occurrence which does not yet have a mate.  Thus the FS-pattern $(F,F,S,S)$ specifies the $(1,2,1,2)$-avoiding clan $(1,2,2,1)$, since when we reach the first S, we insist that it be mated with the more recently appearing (and unmated) $2$ rather than with the $1$.  As another example, the FS pattern $(F,S,+,-,F,F,F,S,+,-,S,F,S,S)$ uniquely determines the $(1,2,1,2)$-avoiding clan $(1,1,+,-,2,3,4,4,+,-,3,5,5,2)$.

This discussion establishes the following.
\begin{theorem}
If $\gamma$ is $(1,2,1,2)$-avoiding, then the clans $\tau$ computed by the above procedure index the irreducible components of the singular locus of $Y_{\gamma}$.
\end{theorem}

\begin{example}
Consider the singular $(4,4)$-clan $\gamma=(1,+,-,2,2,+,-,1)$.  Its path diagram is pictured in Figure \ref{fig:sing-locus-example}, with the four singular corners indicated and numbered.  (Here we draw the boxes of the skew diagram, to make it clearer what is removed to form the new skew/path diagrams.)

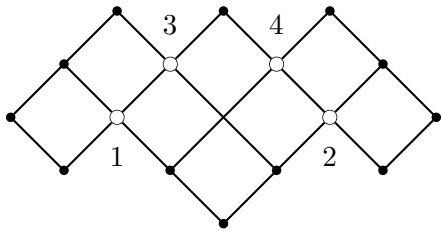
\begin{figure}[h]
\centering
\begin{pspicture}(4,4)
\rput[b]{-45}(0,1.5){
\psdots(0,0)(0,1)(0,2)(1,3)(2,4)(3,4)(4,4)(4,3)(3,2)(3,1)(2,1)(1,0)
\dotnode[dotstyle=o,dotsize=0.2](1,1){A}
\dotnode[dotstyle=o,dotsize=0.2](3,3){B}
\dotnode[dotstyle=o,dotsize=0.2](1,2){C}
\dotnode[dotstyle=o,dotsize=0.2](2,3){D}
\psline(0,0)(0,2)
\psline(1,0)(1,0.90)
\psline(1,1.10)(1,1.90)
\psline(1,2.10)(1,3)
\psline(2,1)(2,2.90)
\psline(2,3.10)(2,4)
\psline(3,1)(3,2.90)
\psline(3,3.10)(3,4)
\psline(4,3)(4,4)
\psline(0,0)(1,0)
\psline(0,1)(0.90,1)
\psline(1.10,1)(3,1)
\psline(0,2)(0.90,2)
\psline(1.10,2)(3,2)
\psline(1,3)(1.90,3)
\psline(2.10,3)(2.90,3)
\psline(3.10,3)(4,3)
\psline(2,4)(4,4)
}
\nput*[labelsep=0.3]{270}{A}{1}
\nput*[labelsep=0.3]{270}{B}{2}
\nput*[labelsep=0.3]{90}{C}{3}
\nput*[labelsep=0.3]{90}{D}{4}
\end{pspicture}
\caption{The path diagram for $(1,+,-,2,2,+,-,1)$}
\label{fig:sing-locus-example}
\end{figure}

The new path diagrams for $\tau_1$, $\tau_2$, $\tau_3$, and $\tau_4$, formed by removing the hook at each of the singular corners $1$, $2$, $3$, and $4$, respectively, appear in Figure \ref{fig:sing-locus-example-2} (read from the top left to the bottom right).  These determine the FS patterns $(+,+,-,F,-,+,-,S)$, $(F,+,-,+,S,+,-,-)$, $(F,S,-,F,+,+,-,S)$, and $(F,+,-,-,S,+,F,S)$, which correspond, respectively, to the clans $\tau_1=(+,+,-,1,-,+,-,1)$, $\tau_2=(1,+,-,+,1,+,-,-)$, $\tau_3=(1,1,-,2,+,+,-,2)$, and $\tau_4=(1,+,-,-,1,+,2,2)$.  Thus the singular locus of $Y_{\gamma}$ is the union of the $K$-orbit closures corresponding to those $4$ clans.

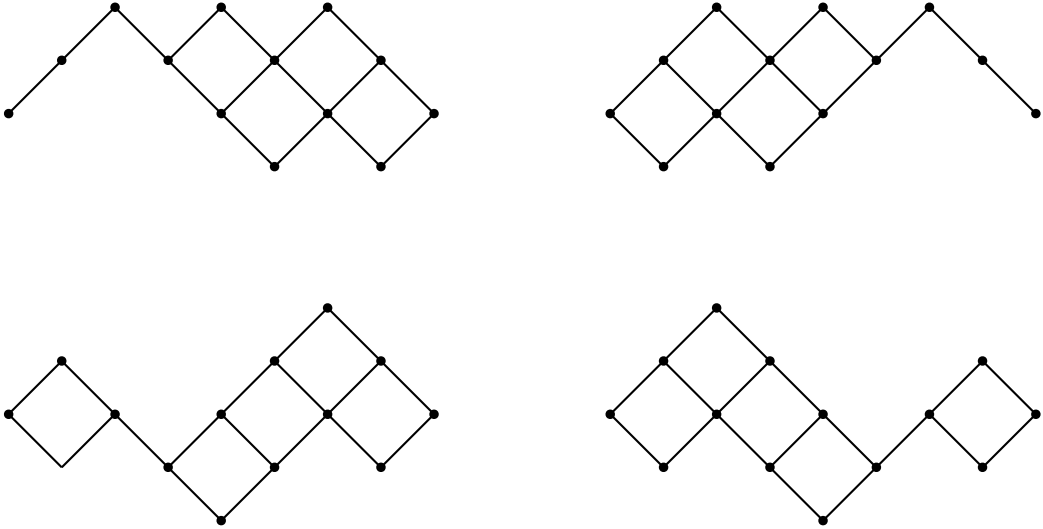
\begin{figure}[h]
\centering
\begin{pspicture}(12,10)
\rput[b](-1,6){
\rput[b]{-45}(0,0){
\psdots(0,0)(0,1)(0,2)(1,2)(1,3)(2,3)(2,4)(3,4)(4,4)(4,3)(3,3)(3,2)(2,2)
\psline(0,0)(0,2)
\psline(1,2)(1,3)
\psline(2,2)(2,4)
\psline(3,2)(3,4)
\psline(4,3)(4,4)
\psline(0,2)(3,2)
\psline(1,3)(4,3)
\psline(2,4)(4,4)
}}

\rput[b](7,6){
\rput[b]{-45}(0,0){
\psdots(0,0)(0,1)(0,2)(1,2)(1,3)(2,3)(2,4)(3,4)(4,4)(2,2)(2,1)(1,1)(1,0)
\psline(0,0)(1,0)
\psline(0,1)(2,1)
\psline(0,2)(2,2)
\psline(1,3)(2,3)
\psline(2,4)(4,4)
\psline(0,0)(0,2)
\psline(1,0)(1,3)
\psline(2,1)(2,4)
}}

\rput[b](-1,2){
\rput[b]{-45}(0,0){
\psdots(0,0)(0,1)(1,1)(2,1)(2,2)(2,3)(2,4)(3,4)(4,4)(4,3)(3,3)(3,2)(3,1)
\psline(0,0)(1,0)
\psline(0,1)(3,1)
\psline(2,2)(3,2)
\psline(2,3)(4,3)
\psline(2,4)(4,4)
\psline(0,0)(0,1)
\psline(1,0)(1,1)
\psline(2,1)(2,4)
\psline(3,1)(3,4)
\psline(4,3)(4,4)
}}

\rput[b](7,2){
\rput[b]{-45}(0,0){
\psdots(0,0)(0,1)(0,2)(1,2)(2,2)(3,2)(3,3)(3,4)(4,4)(4,3)(3,1)(2,1)(1,1)(1,0)
\psline(0,0)(1,0)
\psline(0,1)(3,1)
\psline(0,2)(3,2)
\psline(3,3)(4,3)
\psline(3,4)(4,4)
\psline(0,0)(0,2)
\psline(1,0)(1,2)
\psline(2,1)(2,2)
\psline(3,1)(3,4)
\psline(4,3)(4,4)
}}
\end{pspicture}

\caption{Diagrams for the singular components of $(1,+,-,2,2,+,-,1)$}
\label{fig:sing-locus-example-2}
\end{figure}
\end{example}

\subsection{LCI-ness}\label{sec:lci-ness}
Recall that a local ring $R$ is a \textbf{local complete intersection} (\textbf{lci}) if there exists a regular local ring $S$ and an ideal $I$ generated by a regular sequence on $S$ such that $R \cong S/I$.  A variety or scheme $X$ is said to be \textbf{lci at the point $p$} if the local ring $\caO_{X,p}$ of $X$ at $p$ is lci.  $X$ is simply said to be \textbf{lci} if it is lci at every point.  For any variety $X$, the set of points at which $X$ is \textit{not} lci is a Zariski-closed subset of $X$.

As was the case with smoothness, to understand lci-ness of Richardson varieties, we must know something about the lci locus of the relevant Schubert varieties.  The lci locus of Grassmannian Schubert varieties is now understood by the following proposition, due to C. Darayon~\cite{Darayon-LCIlocus}.  To state the result, we require one further definition.  Given an inner corner on a (top boundary) path diagram, we define the \textbf{left} (respectively \textbf{right}) \textbf{leg length} to be the number of consecutive lattice path segments on the path northwest (respectively northeast) of the corner.

\begin{proposition}
\label{prop:grass-lci-locus}
Let $u\geq v$ be a Grassmannian permutations.  Then $X^v$ is not lci at $u$ if and only if at least one of the following hold:
\begin{enumerate}
\item The path diagram for $v$ has an inner corner not on the path diagram for $u$ that has a left or right leg length greater than 1.
\item The path diagram for $v$ has two consecutive inner corners, neither of which is on the path diagram for $u$.
\end{enumerate}
\end{proposition}

We extend the above definition of leg lengths analogously to singular corners on both the top and bottom boundaries of path diagrams.  Now, combining Lemma~\ref{lem:richardson-p-criterion} and Proposition~\ref{prop:grass-lci-locus}, we have the following:

\begin{proposition}
For Grassmannian $u$ and $v$, the Richardson variety $X_{uw_0^K}^v$ is lci if and only if both of the following hold for every component $C$ of the path diagram:
\begin{enumerate}
	\item Every singular corner of $C$ has both leg lengths equal to $1$.
	\item $C$ contains at most one singular corner on its bottom boundary and at most one singular corner on its top boundary.
\end{enumerate}
\end{proposition}

So, for example, the $K$-orbit $(1,+,-,+,-,1)$ is non-lci, as the lone component of its path diagram, which is shown in Figure \ref{fig:non-lci-bottom}, contains two singular corners along its bottom boundary.  Similarly, $(1,-,+,-,+,1)$ is non-lci, having two singular corners along its top boundary.  In addition, the path diagrams of Figures \ref{fig:singular-top} and \ref{fig:singular-bottom} are both diagrams of non-lci Richardson varieties, since each singular corner has one leg of length $1$ but another of length $2$.

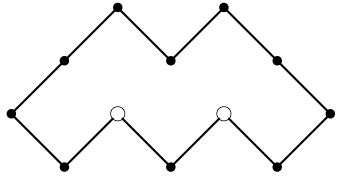
\begin{figure}[h]
\centering
\begin{pspicture}(3,3)
\rput[b]{-45}(0,0.5){
\psdots(0,0)(0,1)(0,2)(1,2)(1,3)(2,3)(3,3)(1,0)(2,1)(3,2)
\psdots[dotstyle=o,dotsize=0.2](1,1)(2,2)
\psline(0,0)(0,2)
\psline(0,2)(1,2)
\psline(1,2)(1,3)
\psline(1,3)(3,3)
\psline(0,0)(1,0)
\psline(1,0)(1,0.92)
\psline(1.08,1)(2,1)
\psline(2,1)(2,1.92)
\psline(2.08,2)(3,2)
\psline(3,2)(3,3)
}
\end{pspicture}
\caption{The path diagram for the non-lci $K$-orbit $(1,+,-,+,-,1)$}
\label{fig:non-lci-bottom}
\end{figure}

As was the case for smoothness, this pictorial requirement on the path diagram for a $(1,2,1,2)$-avoiding clan $\gamma$ can be translated to a pattern avoidance condition on $\gamma$.

\begin{proposition}\label{prop:1212-avoiding-lci}
If $\gamma$ is $(1,2,1,2)$-avoiding, then $Y_\gamma = \overline{Q_{\gamma}}$ is lci if and only $\gamma$ avoids one of $35$ bad patterns.  The bad patterns are the following, along with their negatives (the negative of a pattern being the one obtained from that pattern by inverting all signs):  $(1,+,+,-,1)$, $(1,+,-,-,1)$, $(1,-,2,2,+,1)$, $(1,+,+,2,2,1)$, $(1,+,-,2,2,1)$, $(1,2,2,-,-,1)$, $(1,2,2,+,-,1)$, $(1,2,+,2,-,1)$, $(1,+,2,-,2,1)$, $(1,+,2,3,3,2,1)$, $(1,2,3,3,2,-,1)$, $(1,-,2,2,3,3,1)$, $(1,2,+,2,3,3,1)$ $(1,2,2,+,3,3,1)$, $(1,2,2,3,-,3,1)$, $(1,2,2,3,3,+,1)$, $(1,2,3,3,2,4,4,1)$, $(1,2,2,3,4,4,3,1)$, and $(1,2,2,3,3,4,4,1)$.
\end{proposition}
\begin{proof}
The possible ways for the path diagram for the clan $\gamma$ to have a singular corner with at least one leg length not equal to $1$ are as follows:

\begin{enumerate}
	\item Along the bottom boundary of some component, have three consecutive segments of the form ``up-up-right";
	\item Along the bottom boundary of some component, have three consecutive segments of the form ``up-right-right";
	\item Along the top boundary of some component, have three consecutive segments of the form ``right-up-up"; or
	\item Along the top boundary of some component, have three consecutive segments of the form ``right-right-up".
\end{enumerate}

For the first possibility to occur, there must be three consecutive characters enclosed in an outermost matching pair such that the first two are either $+$'s or second occurrences, and the last is either a $-$ or a first occurrence.  There are $8$ possible sequences, and the minimal clans containing one of these sequences are  $(1,+,+,-,1)$, $(1,+,+,2,2,1)$, $(1,2,+,2,-,1)$, $(1,2,+,2,3,3,1)$, $(1,2,2,+,-,1)$, $(1,2,2,+,3,3,1)$, $(1,2,3,3,2,-,1)$, and $(1,2,3,3,2,4,4,1)$.

For the second possibility, there must be three consecutive characters enclosed in an outermost matching pair such that the first is a $+$ or a second occurrence, and the last two are either $-$'s or first occurrences.  Again, there are $8$ possible sequences, with the list of minimal clans containing one of them as follows:  $(1,+,-,-,1)$, $(1,+,-,2,2,1)$, $(1,+,2,-,2,1)$, ($1,+,2,3,3,2,1)$, $(1,2,2,-,-,1)$, $(1,2,2,-,3,3,1)$, $(1,2,2,3,-,3,1)$, $(1,2,2,3,4,4,3,1)$.

The minimal patterns for possibilities (3) and (4) above are easily seen to be the negatives of these, as was the case when checking smoothness.  This gives us so far $28$ distinct clans.

Now, consider the possible ways for the path diagram to have a component with two singular corners on either its bottom boundary or its top boundary.  We may assume that each singular corner has equal leg lengths of $1$, as otherwise we are already covered by the preceding cases.  Along the bottom boundary, we require a sequence of the form ``up-right-up-right" enclosed in a matching pair.  There are $16$ possible ways this can be accomplished, choosing either a $+$ or a second occurrence for the first position, either a $-$ or a first occurrence for the second position, either a $+$ or a second occurrence for the third position, and either a $-$ or a second occurrence for the fourth position.  Among the $16$ minimal clans containing such a string of consecutive characters within an outermost pair, only four fail to contain one of the bad patterns we have already found.  These are $(1,+,2,2,-,1)$, $(1,+,2,2,3,3,1)$, $(1,2,2,3,3,-,1)$, and $(1,2,2,3,3,4,4,1)$.  (To illustrate the previous point, the clan $(1,+,-,+,-,1)$ is a minimal clan with a sequence of the type we seek, but this clan contains the known bad patterns $(1,+,+,-,1)$ and $(1,+,-,-,1)$, so we do not consider this to be a new bad pattern.)

To handle the top boundary, we simply take the negatives of the patterns causing problems along the bottom boundary, as usual.  This gives three new patterns.  Combining these $3$ with the $4$ listed in the previous paragraph and the first $28$ patterns described above, we have argued that any non-lci clan must contain one of the $35$ bad patterns.

Conversely, suppose that our clan $\gamma$ contains one of the $35$ patterns.  Then we can take the $1$'s of the pattern to be an outermost pair.  Let $C$ be the corresponding component of the path diagram for $\gamma$.  One checks easily that for each of the $35$ patterns given, either the bottom boundary of $C$ contains three (not necessarily consecutive) segments of the form ``up-up-right" or of the form ``up-right-right"; or the top boundary of $C$ contains three (not necessarily consecutive) segments of the form ``right-right-up" or of the form ``right-up-up"; or both.  Considering the bottom boundary, there are two possibilities.  Either the bottom boundary changes from ``up" to ``right" at least twice, giving at least two singular corners on the bottom boundary, or it changes from ``up" to ``right" only once, in which case either the left leg (in the case ``up-up-right") or the right leg (in the case ``up-right-right") has length at least $2$.  In either event, the path diagram must be non-lci.  The reasoning for the top boundary is simply that for the bottom boundary flipped upside down.
\end{proof}

\begin{remark}
We can also compute the non-lci locus of a $(1,2,1,2)$-avoiding $K$-orbit closure, using a similar procedure to that described in Section \ref{sec:sing-locus}.  For each singular corner with at least one leg length larger than 1, we get a component of the non-lci locus by removing the hook for that corner.  Furthermore, if some component has two adjacent singular corners (both on the top or both on the bottom) with all leg lengths 1, we also get a single component of the non-lci locus by removing the hooks for both corners simultaneously.
\end{remark}

Naturally, one would again wonder about the $(1,2,1,2)$-\textit{containing} orbit closures and about combinatorial criteria for determining which ones are lci and which are not.  As we explain below, whether a $K$-orbit closure is lci can be checked by computer in any given example.  With the help of Macaulay 2 \cite{M2} code written by the second author and A. Yong, the authors have been able to study this question experimentally.  Alas, the results of these experiments show that the answer here is not as simple as that of Proposition \ref{prop:1212-rationally-singular}.  Indeed, some $(1,2,1,2)$-containing orbit closures are lci, while others are not.

We describe the mathematics behind using Macaulay 2 \cite{M2} to check lci-ness of a $(1,2,1,2)$-containing clan $\gamma$.  First, recall that, since the non-lci locus of $Y_{\gamma}$ is $K$-stable and closed, it is a union of $K$-orbit closures.  Thus, if $Y_{\gamma}$ is non-lci at some (equivalently, every) point of $Q_{\tau}$, and if $Q_{\delta} \subseteq \overline{Q_{\tau}}$, then $Y_{\gamma}$ is non-lci at every point of $Q_{\delta}$ as well.  So to check whether $Y_{\gamma}$ is globally lci, it suffices to check whether $Y_{\gamma}$ is lci along each \textit{closed} $K$-orbit contained in it.  Moreover, to determine whether $Y_{\gamma}$ is lci along a given closed $K$-orbit, it suffices to check whether $Y_{\gamma}$ is lci at a single point of the closed orbit.

The closed $K$-orbits are parametrized by clans consisting of only $+$'s and $-$'s, and determining whether a given closed orbit $\tau$ is contained in $Y_{\gamma}$ is easy in light of \cite[Theorem 2.5]{Wyser-13}.  Given such a closed orbit, we describe how to determine whether $Y_{\gamma}$ is lci at the distinguished representative $p_v=vB/B$ of $Q_\tau$, which is a $T$-fixed point whenever $Q_\tau$ is closed.  (Indeed, $v$ is the permutation $v(\tau)$, and $p_v$ is represented by the permutation matrix having $1$'s in positions $(v(i),i)$ for each $i=1,\hdots,n$.)

There exists an open affine neighborhood of $Y_{\gamma}$ containing $p_v$ that reflects all the local structure of $Y_{\gamma}$ near $p_v$.  This open neighborhood is called the \textbf{patch} of $Y_{\gamma}$ at $p_v$ (or at $\tau$) in \cite{Wyser-Yong-13}.  The patch is the reduced scheme whose underlying set is $Y_{\gamma} \cap vB^-B/B$, with $vB^-B/B$ a ``permuted big cell".  The permuted big cell is an open affine subset of $G/B$, and it can be coordinatized by representing its general element as a matrix with $1$'s in positions $(v(i),i)$ for each $i=1,\hdots,n$, with $0$'s to the right of these $1$'s, and unspecialized variable entries $z_{i,j}$ in the remaining positions.  We view $vB^-B/B$ as
\[ \mathbb{A}^{\binom{n}{2}} \cong \text{Spec}(R), \]
where $R = \C[\textbf{z}]$, with $\textbf{z}$ the unspecialized $z$-variables mentioned above.  The point $p_v$ corresponds to the origin in $\mathbb{A}^{\binom{n}{2}}$ under this identification.

The patch can be viewed as a reduced and irreducible closed subscheme of $\mathbb{A}^{\binom{n}{2}}$, defined by a prime ideal which we denote by $I_{\gamma,\tau}$.  As explained in \cite{Wyser-Yong-13}, some ``obvious" generators for this ideal are suggested by \cite[Theorem 2.5]{Wyser-13}.  These generators are known to define the patch set theoretically~\cite[Proposition 4.3]{Wyser-Yong-13} and are conjectured~\cite[Conjecture 4.4]{Wyser-Yong-13} to be sufficient to generate all of $I_{\gamma,\tau}$.  Even without a proof of this conjecture, an ideal generated by these obvious equations can certainly be created in Macaulay 2 \cite{M2} in any given example, and its radicalness can be checked, for example by verifying that the generators in question form a Gr\"{o}bner basis with squarefree lead terms with respect to a chosen term order.  Indeed, such checks were the basis of the conjecture in the first place.

Now, note that the patch is stable under the action of $T$, being the intersection of two $T$-stable subsets of $G/B$.  (Any $K$-orbit closure $Y_{\gamma}$ is $T$-stable since it is stable under $K$, which contains the full maximal torus $T$ of $G$.)  Furthermore, there exists a one parameter subgroup $S\subseteq T$ such that the point $p_v$ corresponding to the origin of this affine space is an attractor for $S$.  Hence the $T$-action induces a positive $\Z^n$-grading on $R$ with respect to which $I_{\gamma,\tau}$ is a homogeneous ideal.  Thus one can speak of a minimal free $\Z^n$-graded resolution of $I_{\gamma,\tau}$, or of $R/I_{\gamma,\tau}$.

Now, note that $Y_{\gamma}$ is lci at $p_v$ if and only if the patch $\text{Spec}(R/I_{\gamma,\tau})$ is lci.  Since $I_{\gamma,\tau}$ is homogeneous with respect to a positive grading, the maximal ideal for $p_v$ is the unique maximal graded ideal of $R$, and hence the condition of being a local complete intersection (at $p_v$) is equivalent to that of being a complete intersection.  (See \cite[Prop.1.5.15]{Bruns-Herzog} as well the surrounding section for more details.)  This can be determined by checking whether the codimension of $R/I_{\gamma,\tau}$ is equal to the minimal number of generators of $I_{\gamma,\tau}$.  The latter number can be computed as the first Betti number of the aforementioned minimal free resolution of $R/I_{\gamma,\tau}$.

Given the ideal $I_{\gamma,\tau}$, all of the aforementioned data can be computed using Macaulay 2 \cite{M2}, allowing us to check lci-ness.

\begin{example}\label{ex:m2-lci-checks}
Consider the $(1,2,1,2)$-containing $(3,3)$-clan $\gamma=(1,2,+,-,1,2)$.  The closed orbit $\tau=(+,-,+,+,-,-)$ is contained in the orbit closure $Y_{\gamma}$.  The Macaulay 2 \cite{M2} command

\centerline{\tt 
  I := computePQPatchIdeal("12+-12",3,3,"+-++--"),}
\noindent
developed by the second author and A. Yong, creates the patch ideal $I_{\gamma,\tau}$ described in \cite{Wyser-Yong-13}, as well as giving $\C[\textbf{z}]$ the appropriate $\Z^n$-grading coming from the torus action.  This enables us to then perform the command

\centerline{\tt 
  resl:= prune res((ring I)\textasciicircum 1/I),}
which computes the minimal graded free resolution, followed by

\centerline{\tt 
  betti resl.}
From this command, we learn that the first Betti number is $2$, meaning that $I_{\gamma,\tau}$ is minimally generated by $2$ elements.  On the other hand, the command 

\centerline{\tt 
  codim ((ring I)/I)}
reveals that the patch is of codimension $2$.  Thus the patch of $Y_{\gamma}$ at $\tau$ is a complete intersection, which tells us that $Y_{\gamma}$ is lci along $Q_{\tau}$.  Similarly, one can check that $Y_{\gamma}$ is lci along every closed orbit below it in Bruhat order, so it is in fact globally lci.

On the other hand, performing the same checks for $\gamma=(1,+,2,1,2)$ at $\tau=(+,+,-,-,+)$ reveals that $I_{\gamma,\tau}$ is minimally generated by $3$ elements, but the patch is again of codimension $2$.  So in this case, $Y_{\gamma}$ is \textit{not} lci at $\tau$, so of course it is not globally lci.
\end{example}

The checks described in Example \ref{ex:m2-lci-checks} can be automated to check all $(1,2,1,2)$-containing orbit closures along all closed orbits for a particular $p$ and $q$.  This allows one to give an exhaustive list of which orbit closures are lci and which are not.  The authors have successfully determined the exhaustive list of non-lci $(1,2,1,2)$-containing orbit closures for all $(p,q)$ through $p+q=8$.  As it turns out, through $p+q=8$, lci-ness is characterized by pattern avoidance, even if we allow the $(1,2,1,2)$-containing case.  More precisely, for $p+q \leq 8$, any $K$-orbit closure, $(1,2,1,2)$-avoiding or not, is lci if and only if it avoids the bad patterns of Proposition \ref{prop:1212-avoiding-lci} \textit{and additionally} avoids the patterns $(1,+,2,1,2)$, $(1,2,1,+,2)$, $(1,2,1,3,2,3)$, $(1,2,2,3,1,3)$, $(1,2,1,3,3,2)$, $(1,+,2,3,2,3,1)$, $(1,2,3,2,3,+,1)$, $(1,2,3,2,3,4,4,1)$, $(1,2,2,3,4,3,4,1)$, $(1,2,3,4,2,3,4,1)$, $(1,2,3,4,3,2,4,1)$, and $(1,2,3,4,2,4,3,1)$, along with their negatives.  We do not feel confident in conjecturing that this is the complete list of bad patterns, since we have been unable to push exhaustive checks to $p+q=9$ (or even to the case $K=GL(7,\C) \times GL(2,\C)$) and determine whether any new bad patterns occur there.  However, we feel that with the evidence at hand, considering the results of \cite{Ulfarsson-Woo-11} in the analogous case of Schubert varieties, at least the following conjecture is reasonable.

\begin{conjecture}
LCI-ness of $K$-orbit closures for $(GL(p+q,\C),GL(p,\C)\times GL(q,\C))$ is characterized by pattern avoidance.
\end{conjecture}

\subsection{Gorensteinness}
Recall that a local ring $(R, \mathfrak{m},\mathbbm{k})$ is said to be \textbf{Cohen-Macaulay} if $\text{Ext}_R^i(\mathbbm{k},R) = 0$ for $i < \dim(R)$.  $R$ is said to be \textbf{Gorenstein} if it is Cohen-Macaulay and, additionally, $\text{dim}_{\mathbbm{k}}\text{Ext}_R^{\dim(R)}(\mathbbm{k},R) = 1$.  A variety or scheme $X$ is said to be \textbf{Gorenstein at the point $p$} if the local ring $\caO_{X,p}$ is Gorenstein.  $X$ is said simply to be \textbf{Gorenstein} if it is Gorenstein at every point.  Equivalently, a variety $X$ is Gorenstein if it is Cohen-Macaulay (meaning all local rings are Cohen-Macaulay) and its canonical sheaf is a line bundle.

To determine Gorensteinness of $(1,2,1,2)$-avoiding $K$-orbit closures, we need to determine Gorensteinness of Grassmannian Richardson varieties, which requires an understanding of the Gorenstein locus of a Grassmannian Schubert variety.  Again, this is known, thanks to the following proposition of N. Perrin \cite{Perrin-09}.

\begin{proposition}
The Schubert variety $X^v$ is Gorenstein at $u$ if and only if every inner corner of the path of $v$ not containing a point of $u$ has equal right and left leg lengths.
\end{proposition}

Using Lemma~\ref{lem:richardson-p-criterion}, we can restate Perrin's criterion as follows:

\begin{proposition}\label{prop:gorenstein-test}
Let $u$ and $v$ be Grassmannian permutations.  Then the Richardson variety $X_{uw_0^K}^v$ is Gorenstein if and only if the left leg length equals the right leg length for all singular corners of the path diagram.
\end{proposition}

As examples, the Richardson varieties whose path diagrams are given in Figures \ref{fig:singular-top} and \ref{fig:singular-bottom} are non-Gorenstein, since each has a singular corner with one leg length of $1$ and one leg length of $2$.  On the other hand, the Richardson variety whose path diagram is given in Figure \ref{fig:non-lci-bottom} (which is the closure of the $K$-orbit $(1,+,-,+,-,1)$) \textit{is} Gorenstein, since all three singular corners have common leg length $1$.

\begin{remark}
The non-Gorenstein locus of a non-Gorenstein $(1,2,1,2)$-avoiding $K$-orbit closure is computable by the same procedure as that described in Section \ref{sec:sing-locus} for computing the singular locus, except we should only remove hooks at non-Gorenstein corners, rather than at all singular corners.
\end{remark}

\begin{remark}
It is clear that smooth orbit closures have no singular corners to check, meaning that the tests of Propositions \ref{prop:1212-avoiding-lci} and \ref{prop:gorenstein-test} are passed vacuously.  It is also clear that an orbit closure passing the test of Proposition \ref{prop:1212-avoiding-lci} will also pass the test of Proposition \ref{prop:gorenstein-test}.

Additionally, our described methods of computing the singular locus, the non-lci locus, and the non-Gorenstein locus makes clear that these sets are one containing the next, with the singular locus being the largest set and the non-Gorenstein locus being the smallest.

These simple observations reflect the general fact that ``smooth $\Rightarrow$ lci $\Rightarrow$ Gorenstein".
\end{remark}

For a $(1,2,1,2)$-avoiding clan $\gamma$, one can easily reformulate the check described by Proposition \ref{prop:gorenstein-test} as a check on the clan $\gamma$, without reference to the path diagram.  The statement of the resulting criterion requires a fair amount of notation and is somewhat unpleasant to check, so we do not give it here, since it seems less useful than simply looking at the path diagram for $\gamma$.   We do offer the following example which demonstrates that, unlike the properties of smoothness and lci-ness, Gorensteinness of $(1,2,1,2)$-avoiding orbit closures is \textit{not} characterized by pattern avoidance.

\begin{example}
Consider first the $(1,2,1,2)$-avoiding clan $(1,+,+,-,1)$.  Its path diagram appears in Figure \ref{fig:singular-bottom}.  As mentioned above, this clan is non-Gorenstein.

Now, let  $\gamma'$ be $(1,+,+,-,-,1)$.  Its path diagram is given as Figure \ref{fig:gorenstein}.

\begin{figure}[h]
\centering
\begin{pspicture}(3,3)
\rput[b]{-45}(0,0.5){
\psdots(0,0)(0,1)(0,2)(0,3)(1,3)(2,3)(3,3)(3,2)(2,2)(1,2)(1,1)(1,0)
\psline(0,0)(0,3)
\psline(0,3)(3,3)
\psline(0,0)(1,0)
\psline(1,0)(1,2)
\psline(1,2)(3,2)
\psline(3,2)(3,3)
}
\end{pspicture}
\caption{The Gorenstein $K$-orbit $(1,+,+,-,-,1)$}
\label{fig:gorenstein}
\end{figure}

Note that although $\gamma'$ contains the non-Gorenstein pattern $\gamma$, it \textit{is} Gorenstein, since both legs of its lone singular corner have length $2$.
\end{example}

Again, one would naturally wonder about the $(1,2,1,2)$-containing cases as well.  As was the case with lci-ness, here we are able to check Gorensteinness by computer, as we briefly describe.  As explained in Section \ref{sec:lci-ness}, one can check Gorensteinness of $Y_{\gamma}$ by checking Gorensteinness along every closed orbit contained in $Y_{\gamma}$, and the latter is equivalent to checking Gorensteinness of the patch of $Y_{\gamma}$ taken near a $T$-fixed point of a given closed orbit contained in $Y_{\gamma}$.  This can be checked in Macaulay 2 \cite{M2} using the alternative definition of Gorensteinness given above, namely that a variety is Gorenstein if and only if it is Cohen-Macaulay and its canonical sheaf is a line bundle.  In fact, all patches are automatically Cohen-Macaulay by a general result of Brion \cite[Theorem 6]{Brion-01}.  (We note, though, that Macaulay 2 \cite{M2} \textit{can} easily verify the Cohen-Macaulay property, if only as a sanity check.)  Thus the only check that need be performed is whether the rank of the canonical sheaf is $1$.  This is easily done in Macaulay 2 \cite{M2}.  We demonstrate with two examples.

\begin{example}
Consider first the $(1,2,1,2)$-containing $(3,2)$-clan $\gamma=(1,2,1,+,2)$.  It contains the closed orbit $\tau=(+,-,-,+,+)$.  As before, the patch ideal $I_{\gamma,\tau}$ can be created using the command

\centerline{\tt 
  I := computePQPatchIdeal("121+2",3,2,"+--++"),}
\noindent
and a minimal free $\Z^n$-graded resolution computed as

\centerline{\tt 
  resl:= prune res((ring I)\textasciicircum 1/I).}

The rank of the canonical sheaf is then computed using the command
 
\centerline{\tt 
  rank(source(resl.dd\_cod)),}
where \texttt{cod} is the codimension of $I$, computed as before.  We see here that the canonical sheaf has rank $2$, meaning that $Y_{\gamma}$ is not Gorenstein at $\tau$ and hence is not Gorenstein.

On the other hand, one can perform the same checks for the orbit $\gamma=(1,2,-,1,+,2)$ at every closed orbit contained in it, and one sees that the rank of the canonical sheaf is everywhere $1$, meaning that $Y_{\gamma}$ is Gorenstein.  Note that the Gorenstein pattern $(1,2,-,1,+,2)$ contains the non-Gorenstein pattern $(1,2,1,+,2)$, which shows that Gorensteinness cannot be characterized by pattern avoidance in the $(1,2,1,2)$-containing case either.
\end{example} 

\begin{remark}
Using exhaustive computer checks in the $(1,2,1,2)$-containing cases, we have attempted to find combinatorial criteria generalizing those characterizing Gorensteinness in the $(1,2,1,2)$-avoiding case which would apply equally well to all clans, but we have been unable to do so.
\end{remark}

By analogy with the case of type $A$ Schubert varieties, it is perhaps unsurprising that smoothness and lci-ness (in the $(1,2,1,2)$-avoiding case, and conjecturally in the $(1,2,1,2)$-containing case) can be characterized by pattern avoidance, while Gorensteinness cannot.  Indeed, the pattern avoidance criterion for smoothness of type $A$ Schubert varieties due to Lakshmibai-Sandhya \cite{LS-90} is well-known.  Additionally, the first author and H. {\'U}lfarrson \cite{Ulfarsson-Woo-11} have recently shown that lci-ness of Schubert varieties can also be characterized by pattern avoidance.  However, the first author and A. Yong showed in \cite{Woo-Yong-06} that Gorensteinness of Schubert varieties \textit{cannot} be characterized by ordinary pattern avoidance.  A more general notion of ``Bruhat-restricted pattern avoidance" is needed.  This was later shown in \cite{Woo-Yong-08} to be an example of the yet more general notion of ``interval pattern avoidance".  It is reasonable to wonder whether there is a similar generalization of pattern avoidance in the $K$-orbit setting.

\begin{question}
Is there a combinatorial notion in the $K$-orbit setting, analogous to interval pattern avoidance in the case of Schubert varieties, which explains or ``governs" all semicontinuously stable singularity properties of $K$-orbit closures?
\end{question}

The above question is natural not only in the case of $(G,K) = (GL(p+q,\C),GL(p,\C) \times GL(q,\C))$ that we are currently considering, but also for other symmetric pairs such as $(GL(2n,\C),Sp(2n,\C))$ and $(GL(n,\C),O(n,\C))$, where combinatorial parametrizations of the orbit sets are known, and where there are reasonable corresponding notions of pattern avoidance.

Furthermore, Billey and Braden~\cite{Billey-Braden} (partially anticipated by Bergeron and Sottile~\cite{Bergeron-Sottile-98}) give a geometric explanation for the appearance of pattern avoidance in characterizing smoothness of Schubert varieties.  Billey and Postnikov~\cite{Billey-Postnikov-05} give a uniform definition of pattern avoidance in terms of the underlying root systems that matches this geometric explanation and use it to characterize smoothness for arbitrary Schubert varieties.\footnote{The paper~\cite{Billey-Postnikov-05} actually precedes~\cite{Billey-Braden}, but was published later due to delays in the publication process.} This definition was extended to interval pattern avoidance by the first author~\cite{Woo-10}.  We therefore also ask the following.

\begin{question}
Is there a uniform combinatorial notion in the $K$-orbit setting (at least applying to the case of simply connected $G$) that gives a geometric explanation for the appearance of pattern avoidance in smoothness results?  Can this notion also be extended to a notion of interval pattern avoidance?
\end{question}

\section{Formulas for KLV Polynomials}

In this section, we give an explicit formula for KLV polynomials $P_{\tau,\gamma}(q)$ when $\gamma$ is $(1,2,1,2)$-avoiding.

A very brief sketch of the proof, which we flesh out in more detail in the coming sections, is as follows:  $Y_{\gamma}$ is the Richardson variety $X_{w_0^Ku(\gamma)^{-1}}^{v(\gamma)^{-1}}$ by Theorem \ref{thm:richardson-theorem}, and the KLV polynomial $P_{\tau,\gamma}(q)$ is the local intersection homology ($IH$) Poincar\'{e} polynomial for $Y_{\gamma}$ at a point of $Q_{\tau}$, as we saw in Section \ref{ssec:geometric-klv}.  By results of \cite{Knutson-Woo-Yong-13}, the $IH$ Poincar\'{e} polynomial for a Richardson variety is a product of those for the individual Schubert varieties, each of which is an ordinary KL polynomial.  Finally, the KL polynomials in question are those for Grassmannian Schubert varieties, and an explicit formula for such KL polynomials is given explicitly by Lascoux--Sch\"{u}tzenberger in \cite{LS-81}.

\subsection{The Lascoux--Sch\"{u}tzenberger formula for cograssmannian KL polynomials}\label{ssec:ls-theorem}
We start by describing the aforementioned rule of Lascoux--Sch\"utzenberger for calculating $P_{v,w}(q)$ when $w$ is cograssmannian. Our account will be closer in spirit to that of \cite{JW-13} than to the original~\cite{LS-81}.

Let $w \in S_n$ be a \textbf{cograssmannian} permutation, meaning one with at most one right ascent.  Let $p$ be the location of this ascent.  Just as we did with Grassmannian permutations in Section \ref{sec:path-criteria}, we can associate to $w$ a lattice path in a $p \times (n-p)$ rectangle, going from the southwest corner to the northeast, by moving up at the $i$-th step if $w^{-1}(i) \leq p$ and moving right otherwise.  Note that this is the same as the path for the Grassmannian permutation $ww_0^K$, with $w_0^K$ the long element of $W_K = S_p \times S_{n-p}$.

Following \cite{LS-81}, we use this lattice path to associate a rooted tree $\mathcal{T}(w)$ to $w$.  First write a string of parentheses for the lattice path, replacing up steps with ``('' and right steps with ``)``.  Each matching pair of parentheses ``($\ldots$)'' will be a vertex of the tree, and vertex $V$ is a descendent of $V'$ if the parentheses corresponding to $V$ are nested inside the parentheses corresponding to $V'$.  Finally, add a root vertex (not corresponding to any parentheses) which will be the ancestor of every vertex.

Given any permutation $x$, we associate a path to $x$ in the same way by drawing a lattice path in the $p \times (n-p)$ rectangle from the southwest corner to the northeast corner by the same rules.  (One can do this even if $x$ is neither Grassmannian nor cograssmannian.  This process recovers the path for $x'$, where $x'$ is the maximal length coset representative for $xW_K$.  It is well known that in this case $P_{x,w}(q)=P_{x',w}(q)$.)

To each leaf of $\mathcal{T}(w)$, we assign a nonnegative integer $c$ based on $x$, which we will call the {\bf capacity} of the leaf.  A leaf of $\mathcal{T}(w)$ is associated to a consecutive pair ``()'' of parentheses, which corresponds to an inner corner of $w$'s path.  Say that this inner corner is at $(a,b)$.  Assuming $x \leq w$, so that the path for $w$ lies weakly southeast of that for $x$, let the capacity $c$ of the corresponding leaf be the unique nonnegative integer for which $(a-c,b+c)$ is on the path for $x$.  Pictorially, in our $45^{\circ}$-rotated path diagrams, the capacity is easy to read off as the vertical distance from an inner corner on the lower path to the upper path( note that this capacity is nonzero if and only if the inner corner is a ``singular corner" by our earlier definition).

As an example, let $w = 986517432$, and let $x$ be any permutation in the same left $W_K$ coset as $764219853$.  The rotated path diagram is given in Figure \ref{fig:first-rotated-diagram}.  We reproduce it here as Figure \ref{fig:first-LS-diagram} with the inner corners on the bottom path indicated, as well as their vertical distances to the upper path.

\begin{figure}[h]
\centering
\begin{pspicture}(4,5)
\rput[b]{-45}(0,2.5){
\psdots(0,0)(0,2)(1,2)(1,3)(2,3)(2,5)(3,5)(4,5)(1,1)(2,1)(3,1)(3,2)(4,3)(4,4)
\dotnode[dotstyle=o,dotsize=0.195](0,1){A}
\dotnode[dotstyle=o,dotsize=0.195](3,3){B}
\dotnode(2,4){C}
\psline(0,0)(0,0.90)
\psline(0,1.10)(0,2)
\psline(0,2)(1.05,2)
\psline(1,2)(1,3.05)
\psline(1,3)(2.05,3)
\psline(2,3)(2,5.05)
\psline(2,5)(4.05,5)
\psline(0.10,1)(3.05,1)
\psline(3,1)(3,2.90)
\psline(3.10,3)(4.05,3)
\psline(4,3)(4,5.05)
}
\ncline[linestyle=dotted]{B}{C}
\ncput*{1}
\nput*[labelsep=0.2]{175}{A}{0}
\end{pspicture}
\caption{The path diagram for $w=986517432$, $x=764219853$}
\label{fig:first-LS-diagram}
\end{figure}
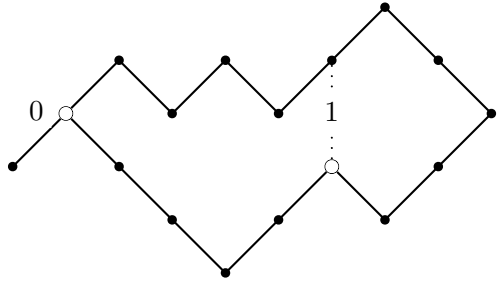

The word corresponding to $w$ is ``()))(()((".  The corresponding rooted tree, with leaf capacities indicated, is shown in Figure \ref{fig:first-tree}.

\begin{figure}[h]
\centering
\begin{pspicture}(6,4)
\psdots(3,0)(3,1)(2,2)(1,3)(4,2)
\dotnode(0,4){A}
\dotnode(5,3){B}
\psline(3,0)(3,1)
\psline(3,1)(0,4)
\psline(3,1)(5,3)
\nput*[labelsep=0.2]{175}{A}{0}
\nput*[labelsep=0.2]{5}{B}{1}
\end{pspicture}
\caption{The rooted tree for $w=986517432$, $x=764219853$}
\label{fig:first-tree}
\end{figure}
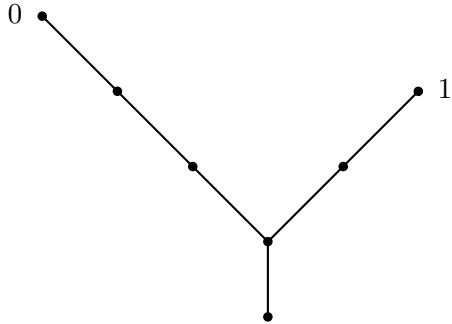

Now let $A_{x,w}$ be the set of edge labellings of $\mathcal{T}(w)$ with entries in $\mathbb{Z}_{\geq0}$ satisfying the following:
\begin{enumerate}
\item Labels weakly increase along any path from the root to a leaf.
\item The label on any edge adjacent to a leaf does not exceed the capacity of the leaf.
\end{enumerate}

Given $t\in A_{x,w}$, let $|t|$ denote the sum of the labels.  Then the formula of \cite{LS-81} is as follows:
\begin{theorem}[\cite{LS-81}]\label{thm:LS-KL}
For any cograssmannian $w \in S_n$ and for any $x$,
\[ P_{x,w}(q)=\sum_{t\in A_{x,w}} q^{|t|}. \]
\end{theorem}

As an example, there are three valid edge labellings of the rooted tree shown in Figure \ref{fig:first-tree}.  Each of them is shown in Figure \ref{fig:edge-labellings}.  Thus Theorem \ref{thm:LS-KL} tells us that
\[ P_{764219853,986517432}(q) = 1 + q + q^2. \]

\begin{figure}[h]
\centering
\begin{pspicture}(20,4)
\dotnode(1,0){A}
\dotnode(1,1){B}
\dotnode(0,2){C}
\dotnode(-1,3){D}
\dotnode(-2,4){E}
\dotnode(2,2){F}
\dotnode(3,3){G}
\dotnode(8,0){H}
\dotnode(8,1){I}
\dotnode(7,2){J}
\dotnode(6,3){K}
\dotnode(5,4){L}
\dotnode(9,2){M}
\dotnode(10,3){N}
\dotnode(15,0){O}
\dotnode(15,1){P}
\dotnode(14,2){Q}
\dotnode(13,3){R}
\dotnode(12,4){S}
\dotnode(16,2){T}
\dotnode(17,3){U}
\ncline{A}{B}
\naput*{0}
\ncline{B}{C}
\naput*{0}
\ncline{C}{D}
\naput*{0}
\ncline{D}{E}
\naput*{0}
\ncline{B}{F}
\nbput*{0}
\ncline{F}{G}
\nbput*{0}
\ncline{H}{I}
\naput*{0}
\ncline{I}{J}
\naput*{0}
\ncline{J}{K}
\naput*{0}
\ncline{K}{L}
\naput*{0}
\ncline{I}{M}
\nbput*{0}
\ncline{M}{N}
\nbput*{1}
\ncline{O}{P}
\naput*{0}
\ncline{P}{Q}
\naput*{0}
\ncline{Q}{R}
\naput*{0}
\ncline{R}{S}
\naput*{0}
\ncline{P}{T}
\nbput*{1}
\ncline{T}{U}
\nbput*{1}
\end{pspicture}
\caption{Valid edge labellings for $w=986517432$, $x=764219853$}
\label{fig:edge-labellings}
\end{figure}
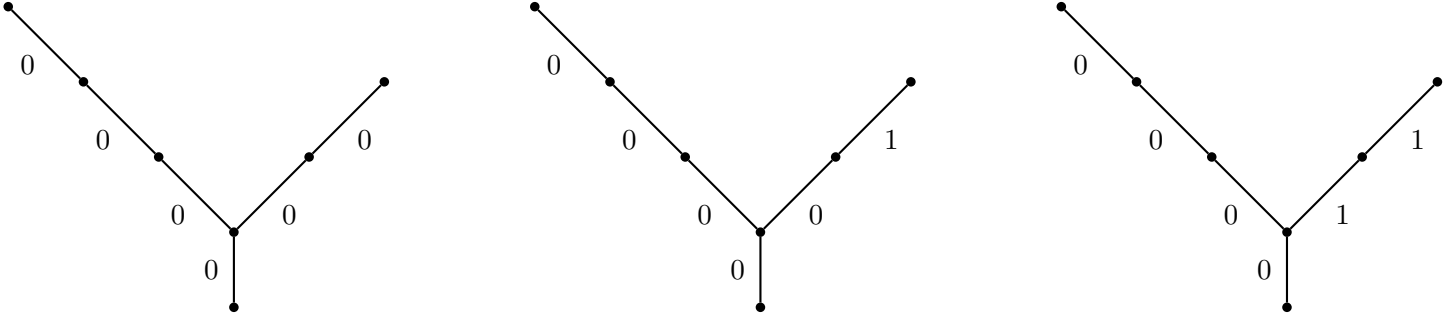

\begin{remark}\label{rmk:upside-down-ls}
Note that Theorem \ref{thm:LS-KL} can also directly calculate $P_{w_0w, w_0v}(q)$ when $v$ is a Grassmannian permutation, and $w$ is arbitrary.  Indeed, when $v$ is a Grassmannian permutation with descent at $p$, then $w_0v$ is a cograssmannian permutation with descent at $q=n-p$.  Thus Theorem \ref{thm:LS-KL} applies.  One can associate to $v$ a lattice path in the $p \times q$ grid as described in Section \ref{sec:path-criteria}.  We can then directly construct $\mathcal{T}(w_0v)$ by matching parentheses the the ``wrong'' way, so a vertex corresponds to a matching ``)$\ldots$('' pair, and a leaf corresponds to an inner corner of $v$'s path.  Similarly, we can assign capacities to $\mathcal{T}(w_0v)$ corresponding to $w_0w$ by constructing a path for $w$ (now southeast of that for $v$) and looking at the distance from each inner corner of $v$'s path to $w$'s path.  Essentially, this amounts to doing the computation for $P_{w_0w,w_0v}(q)$ (which would ordinarily be done in a $q \times p$ rectangle) after rotating the diagrams by $180^{\circ}$, which is equivalent to flipping the path diagrams upside down and switching left and right.
\end{remark}

As an example of this last observation, consider the case $w=156892347$ and $v=124673589$.  Then the appropriate diagram for calculating $P_{w_0w, w_0v}(q)$ is that given in Figure \ref{fig:first-LS-diagram}, but here we consider different corners (the outer corners along the top path) and their vertical distances to the lower path, as indicated in Figure \ref{fig:second-LS-diagram}.  The word for the top path is ``(()()(())", and, matching the wrong way as described above, we obtain the rooted tree shown in Figure \ref{fig:second-tree}.

\begin{figure}[h]
\centering
\begin{pspicture}(4,5)
\rput[b]{-45}(0,2.5){
\psdots(0,0)(0,1)(0,2)(1,3)(2,5)(3,5)(4,5)(1,1)(3,1)(4,3)(4,4)
\dotnode[dotstyle=o,dotsize=0.195](1,2){A}
\dotnode[dotstyle=o,dotsize=0.195](2,3){B}
\dotnode(2,1){C}
\dotnode(3,2){D}
\psline(0,0)(0,2)
\psline(0,2)(0.90,2)
\psline(1,2.10)(1,3)
\psline(1,3)(1.90,3)
\psline(2,3.10)(2,5)
\psline(2,5)(4,5)
\psline(0,1)(3,1)
\psline(3,1)(3,3)
\psline(3,3)(4,3)
\psline(4,3)(4,5)
}
\ncline[linestyle=dotted]{A}{C}
\ncput*{1}
\ncline[linestyle=dotted]{B}{D}
\ncput*{1}
\end{pspicture}
\caption{The path diagram for $v=124673589$, $w=156892347$}
\label{fig:second-LS-diagram}
\end{figure}
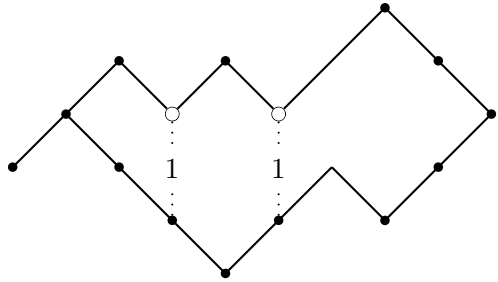

\begin{figure}[h]
\centering
\begin{pspicture}(4,3)
\psdots(2,0)(2,1)(1,2)(3,2)
\dotnode(0,3){A}
\dotnode(4,3){B}
\psline(2,0)(2,1)
\psline(2,1)(0,3)
\psline(2,1)(4,3)
\nput*[labelsep=0.2]{175}{A}{1}
\nput*[labelsep=0.2]{5}{B}{1}
\end{pspicture}
\caption{The rooted tree for $v=124673589$, $w=156892347$}
\label{fig:second-tree}
\end{figure}
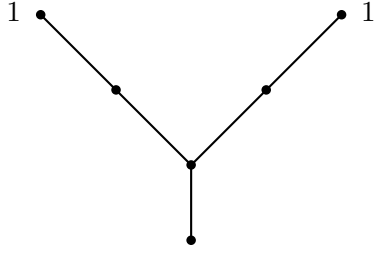

One checks that this tree has $10$ valid edge labellings and that Theorem \ref{thm:LS-KL} gives the KL polynomial
\[ P_{954218763,986437521}(q) = 1+2q+3q^2+2q^3+q^4+q^5. \]

\subsection{Statement and proof of the theorem}
We can now give the statement and proof of our formula for the KLV polynomial $P_{\tau,\gamma}(q)$ for $(1,2,1,2)$-avoiding clans $\gamma$.  

We start with the path diagram corresponding to the clan $\gamma$, described in Section \ref{sec:path-criteria}.  We then add to it another path diagram for the clan $\tau$, constructed in the same way that the path diagram for $\gamma$ is.  Namely, we start from the southwest of the $p \times q$ grid and trace two paths, following these rules at step $i$:
\begin{enumerate}
	\item If $\tau_i=+$, then both paths move up;
	\item If $\tau_i=-$, then both paths move right;
	\item If $\tau_i$ is a first occurrence, then the bottom path moves right, while the top path moves up; and
	\item If $\tau_i$ is a second occurrence, the bottom path moves up, while the top path moves right.
\end{enumerate}

Note that this makes sense even if $\tau$ is not $(1,2,1,2)$-avoiding, and if $\tau \leq \gamma$, the new path diagram for $\tau$ fits within that for $\gamma$.  (This follows easily from \cite[Theorem 2.5]{Wyser-13}.)  Note further that the top path of the new path diagram for $\tau$ is that for the Grassmannian permutation $v(\tau)$, while the bottom path of the new diagram is that for the permutation $u(\tau)$, where $v(\tau)$ and $u(\tau)$ are as described in Section \ref{ssec:clans}.

We next construct two rooted trees in the way just described in Section \ref{ssec:ls-theorem}:  For the first, we construct a word in `(' and `)' for the bottom path, and a rooted tree from that word by matching ``()", with leaf capacities determined by the distances from singular corners on the bottom path for $\gamma$ to the bottom path for $\tau$.  For the second, we construct a word for the \textit{top} path, and a rooted tree from that word by matching ``)(", with leaf capacities determined by the distances from singular corners on the top path for $\gamma$ to the top path for $\tau$.

Define the set $\mathcal{T}_{\tau,\gamma}$ as the set of all pairs consisting of a valid edge labelling of the first rooted tree, combined with a valid edge labelling of the second, with valid edge labellings defined just as in Section \ref{ssec:ls-theorem}.  For such a pair $t$, define $|t|$ to be the sum of labels (taken over both edge labellings).  Then we have the following theorem.

\begin{theorem}\label{thm:klv-formula}
Given clans $\tau \leq \gamma$ in Bruhat order, with $\gamma$ avoiding $(1,2,1,2)$, the KLV polynomial satisfies
\[ P_{\tau,\gamma}(q)=\sum_{t\in\mathcal{T}_{\tau,\gamma}} q^{|t|}. \]
\end{theorem}
\begin{proof}
The following Kunneth formula is given in \cite{GM-83}:  Given two stratified pseudomanifolds $X$ and $Y$, the intersection homology sheaf (assuming the middle perversity) is given by
\[ \mathcal{IH}(X\times Y) \cong \mathcal{IH}(X)\otimes\mathcal{IH}(Y). \]

Now, given points $x\in X$ and $y\in Y$, localizing at $p=(x,y)\in X\times Y$ tells us that 
\[ \mathcal{IH}_{p}(X \times Y)=\mathcal{IH}_x(X) \otimes \mathcal{IH}_y(Y). \]

Now the usual Kunneth formula tells us (since we are working with sheaves of vector spaces, so there are no flatness issues to worry about) that
\[ IH^i_p(X\times Y)\cong \bigoplus_{j+k=i} IH^j_x(X)\otimes IH^k_y(Y). \]

We apply this to a Richardson variety $X_u^v$, utilizing results of \cite{Knutson-Woo-Yong-13}.  Since, for $p\in\mathbb{C}^n$ (or actually any smooth space), $IH^k_p(\mathbb{C}^n)=0$ for $k>0$ and $\dim IH^0_p(\mathbb{C}^n)=1$, an application of \cite[Theorem 1.1]{Knutson-Woo-Yong-13} gives the following:

\begin{theorem}\label{thm:kwy-ih-theorem}
Let $p\in X_u^v$, and suppose $p\in X_y^\circ \cap X_\circ^w$.  Then
\[ IH^i_p(X_u^v) \cong \bigoplus_{j+k=i} IH^j_{yB/B}(X_u)\otimes IH^k_{wB/B}(X^v). \]
Consequently,
\[ \sum_{i} \dim(IH^i_p(X_u^v)) q^{i/2} = P_{w_0w,w_0v}(q)P_{y,u}(q). \]
\end{theorem}

Now, by Theorem \ref{thm:richardson-theorem}, if $\gamma$ avoids $(1,2,1,2)$, $Y_\gamma$ is the Richardson variety $X_u^v$ where $u=w_0^Ku(\gamma)^{-1}$, and $v = v(\gamma)^{-1}$.  By the results of Section \ref{ssec:geometric-klv}, the KLV polynomial $P_{\tau,\gamma}$ is the $IH$-Poincar\'{e} polynomial for $Y_{\gamma}$ at a point of $Q_{\tau}$, and Theorem \ref{thm:kwy-ih-theorem} applies to this computation directly.   Recall that Lemma \ref{lem:t-fixed-representative} tells us that the distinguished representative of $Q_{\tau}$ lies in $X_u^0 \cap X_0^v$, where $v=v(\tau)^{-1}$, and where $u^{-1}=u(\tau)w$ for some $w \in W_K$.  Thus Theorem \ref{thm:kwy-ih-theorem} says that
\[ P_{\tau,\gamma}(q) = P_{w_0v, w_0v(\gamma)^{-1}}(q) P_{u,w_0^Ku(\gamma)^{-1}}(q) = P_{w_0v(\tau),w_0v(\gamma)}(q) P_{u(\tau)w,u(\gamma)w_0^K}(q) = \]
\[ P_{w_0v(\tau),w_0v(\gamma)}(q) P_{u(\tau),u(\gamma)w_0^K}(q), \]
where we have used the following standard facts:
\begin{enumerate}
	\item $P_{a,b}(q) = P_{a^{-1},b^{-1}}(q)$ for any $a,b \in W$; and
	\item If $u$ is cograssmannian and $x,x' \leq u$ are such that $x' = xw$ for $w \in W_K$, then $P_{x',u}(q) = P_{x,u}(q)$.
\end{enumerate}

Now, recall that $u(\gamma)$ and $v(\gamma)$ are Grassmannian, so $u(\gamma)w_0^K$ and $w_0v(\gamma)$ are cograssmannian.  Thus each of the polynomials in the expression above is computed by Theorem \ref{thm:LS-KL}.  (We use Remark \ref{rmk:upside-down-ls} to compute $P_{w_0v(\tau),w_0v(\gamma)}(q)$.)  As mentioned, the path diagram for $\tau$ consists of the path for $v(\tau)$ on top and the path for $u(\tau)$ on bottom.  The polynomial described by our theorem is thus the appropriate product of KL polynomials.
\end{proof}

\begin{example}
We compute the KLV polynomial $P_{\tau,\gamma}(q)$ where $\gamma=(1,2,+,-,+,-,2,1)$, and $\tau=(+,-,+,-,+,-,+,-)$.  The path diagram for $\gamma$, with the added (dashed) path diagram for $\tau$, is given in Figure \ref{fig:klv-path-diagram}.

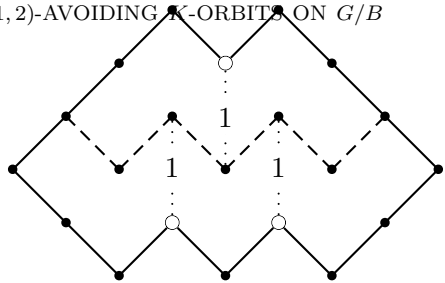
\begin{figure}[h]
\centering
\begin{pspicture}(4,4)
\rput[b]{-45}(0,2.5){
\psdots(0,0)(0,1)(0,2)(0,3)(1,4)(2,4)(3,4)(4,4)(4,3)(4,2)(3,1)(2,0)(1,0)(1,1)(3,3)
\dotnode[dotstyle=o,dotsize=0.195](1,3){A}
\dotnode[dotstyle=o,dotsize=0.195](2,1){B}
\dotnode[dotstyle=o,dotsize=0.195](3,2){C}
\dotnode(1,2){D}
\dotnode(2,2){E}
\dotnode(2,3){F}
\psline(0,0)(0,3)
\psline(0,3)(0.90,3)
\psline(1,3.10)(1,4)
\psline(1,4)(4,4)
\psline(4,4)(4,2)
\psline(2,0)(2,0.90)
\psline(2.10,1)(3,1)
\psline(3,1)(3,1.90)
\psline(3.10,2)(4,2)
\psline(2,0)(0,0)
\psline[linestyle=dashed](0,1)(1,1)
\psline[linestyle=dashed](1,1)(1,2)
\psline[linestyle=dashed](1,2)(2,2)
\psline[linestyle=dashed](2,2)(2,3)
\psline[linestyle=dashed](2,3)(3,3)
\psline[linestyle=dashed](3,3)(3,4)
}
\ncline[linestyle=dotted]{A}{E}
\ncput*{1}
\ncline[linestyle=dotted]{B}{D}
\ncput*{1}
\ncline[linestyle=dotted]{C}{F}
\ncput*{1}
\end{pspicture}
\caption{The path diagram for $\tau=(+,-,+,-,+,-,+,-)$, $\gamma=(1,2,+,-,+,-,2,1)$}
\label{fig:klv-path-diagram}
\end{figure}

The rooted trees and leaf capacities used to calculate $P_{u(\tau),u(\gamma)w_0^K}(q) = P_{13572468,87536421}(q)$ and $P_{w_0v(\tau),w_0v(\gamma)}(q)=P_{86427531,87645321}(q)$ are
given as Figure \ref{fig:klv-trees}.

\begin{figure}[h]
\centering
\begin{pspicture}(6,4)
\psdots(1,0)(5,0)
\dotnode(0,1){A}
\dotnode(2,1){B}
\dotnode(5,1){C}
\psline(1,0)(0,1)
\psline(1,0)(2,1)
\psline(5,0)(5,1)
\nput*[labelsep=0.2]{175}{A}{1}
\nput*[labelsep=0.2]{5}{B}{1}
\nput*[labelsep=0.2]{5}{C}{1}
\end{pspicture}
\caption{The rooted trees for $P_{\tau,\gamma}(q)$}
\label{fig:klv-trees}
\end{figure}
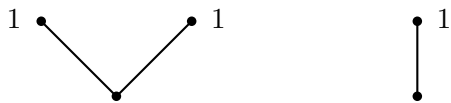

There are $4$ valid edge labellings of the first tree, and $2$ of the second, giving the following two ordinary KL polynomials:

\[ P_{13572468,87536421}(q) = 1 + 2q + q^2, \]
and
\[ P_{86427531,87645321}(q) = 1 + q. \]

Thus
\[ P_{\tau,\gamma}(q) = (1+2q+q^2)(1+q) = 1 + 3q + 3q^2 + q^3. \]
\end{example}

\section*{Acknowledgements}
We gratefully acknowledge Alexander Yong for helpful discussions, as well as for allowing our use of Macaulay 2 \cite{M2} code developed jointly by him and the second author for the purpose of furthering this project.  We also wish to thank Peter Trapa and Monty McGovern for their generous help with various technical details.

\bibliographystyle{plain}
\bibliography{../sourceDatabase}

\end{document}